\documentclass[11pt]{article}
\usepackage{amsmath,amsfonts,amssymb,amsthm,fancyhdr,bm,mathtools,enumitem}
\usepackage[hyphens]{url}
\usepackage[hidelinks]{hyperref}

\usepackage{fullpage}
\usepackage[dvipsnames]{xcolor}
\usepackage[capitalize]{cleveref}
\usepackage[color=Purple!50!white]{todonotes}
\usepackage[numbers,sort]{natbib}
\usepackage{tikz,ifthen}

\usetikzlibrary{decorations.markings,arrows.meta}
\tikzset{vert/.style={draw, fill=black, circle, inner sep=2pt}}

\newtheorem{theorem}{Theorem}[section]
\newtheorem{lemma}[theorem]{Lemma}
\newtheorem{proposition}[theorem]{Proposition}

\newtheorem{conjecture}[theorem]{Conjecture}
\newtheorem{question}[theorem]{Question}

\theoremstyle{definition}
\newtheorem{definition}[theorem]{Definition}

\crefname{equation}{equation}{equations}
\crefname{lemma}{Lemma}{Lemmas}
\crefname{proposition}{Proposition}{Propositions}
\crefname{claim}{Claim}{Claims}
\crefname{theorem}{Theorem}{Theorems}
\crefname{conjecture}{Conjecture}{Conjectures}
\crefname{figure}{Figure}{Figures}

\newlist{lemenum}{enumerate}{1}
\setlist[lemenum]{label=(\alph*), ref=\thelemma(\alph*)}
\crefalias{lemenumi}{lemma}

\DeclareMathOperator*\pr{Pr}

\DeclareMathOperator\twr{twr}
\DeclareMathOperator\wow{wow}

\newcommand\up[1]{^{(#1)}}
\newcommand\wh[1]{\widehat{#1}}
\newcommand\wt[1]{\widetilde{#1}}
\newcommand\ab[1]{\lvert#1\rvert}

\newcommand{\F}{\mathcal{F}}
\newcommand{\R}{\mathbb{R}}
\newcommand{\N}{\mathbb{N}}
\newcommand{\V}{\mathcal{V}}
\newcommand{\Q}{\mathcal{Q}}
\newcommand{\p}{\mathcal{P}}
\newcommand{\h}{\mathcal{H}}
\newcommand\1{\mathbf{1}}

\newcommand\symmd{\mathbin{\triangle}}
\newcommand\dtimes{\mathbin{\times \dotsb \times}}

\DeclareMathOperator*{\E}{\mathbb{E}}

\DeclareMathOperator{\poly}{poly}

\let\leq\leqslant
\let\geq\geqslant

\title{Regularity for hypergraphs with bounded VC\texorpdfstring{$_2$}{\_2} dimension}
\author{Lior Gishboliner\thanks{Department of Mathematics, University of Toronto, ON, Canada. {Email}: \texttt{lior.gishboliner@utoronto.ca}. Research supported by the NSERC Discovery Grant ``Problems in Extremal and Probabilistic Combinatorics".}\and Asaf Shapira\thanks{School of Mathematics, Tel Aviv University, Tel Aviv 69978, Israel. Email: \texttt{asafico@tau.ac.il}. Supported in part by ERC Consolidator Grant 863438.} \and Yuval Wigderson\thanks{Institute for Theoretical Studies, ETH Z\"urich, 8006 Z\"urich, Switzerland. Email: \texttt{yuval.wigderson@eth-its. ethz.ch}. Research supported by Dr.\ Max R\"ossler, the Walter Haefner Foundation, and the ETH Z\"urich Foundation.}}

\date{}

\begin{document}
\maketitle
\begin{abstract}
While Szemer\'edi's graph regularity lemma is an indispensable tool for studying extremal problems in graph theory, using
it comes with a hefty price, since in the worst case, a graph may only have regular partitions of tower-type order. It is thus
sensible to ask if there is some natural restriction which forces graphs to have much smaller regular partitions.
A celebrated result of this type, due to Alon--Fischer--Newman and Lov\'asz--Szegedy, states that for graphs of bounded VC
dimension, one can reduce the tower-type bounds to polynomial.

The graph regularity lemma has been extended to the setting of $k$-graphs by Gowers and by R\"odl et al.
Unfortunately, these lemmas come with even larger Ackermann-type bounds. Chernikov--Starchenko and Fox--Pach--Suk considered a strong notion
of $k$-graph VC dimension and proved that $k$-graphs of bounded VC dimension have regular partitions of polynomial size.
Shelah introduced a weaker and combinatorially more natural notion of dimension, called VC$_2$ dimension, which has been
extensively studied in the past decade. In particular, Chernikov, Towsner, Terry, and Wolf asked if one can improve the worst case bounds for $3$-graph regularity when the $3$-graph
has bounded VC$_2$ dimension. 
Our main result in this paper answers this question positively in the following strong sense: 
we show that in the setting of bounded VC$_2$ dimension, one can reduce the bounds for 3-graph regularity by 
one level in Ackermann hierarchy. Furthermore, our new improved bound is best possible.

Our proof has two key steps. We first introduce a novel method for designing regularity lemmas for graphs of bounded VC dimension, based on the cylinder regularity lemma. We then prove a hypergraph version of the cylinder regularity lemma, which allows us to extend this new method to hypergraphs. We additionally highlight a few other applications of this cylinder regularity lemma, which we expect to find many other uses.

\end{abstract}

\section{Introduction}
\subsection{Background}
\subsubsection{Graph regularity}\label{sec:graph regularity}
One of the deepest and most powerful results in modern graph theory is Szemer\'edi's regularity lemma \cite{MR0540024}. Roughly speaking, it states that any graph can be approximated as the union of a constant number of ``random-like'' graphs, where the number of such pieces depends only on the quality of the approximation and on how random-like we demand the pieces to be. This foundational result has found many applications in graph theory, number theory, theoretical computer science, and beyond.

To state (one version of) the regularity lemma, let $X,Y$ be vertex subsets of some graph $G$. We denote by $e(X,Y)$ the number of pairs in $X \times Y$ that are edges of $G$, and by $d(X,Y) \coloneqq e(X,Y)/(\ab X \ab Y)$ the \emph{edge density} of the pair. We say that the pair $(X,Y)$ is \emph{$\varepsilon$-quasirandom} if, for all $A \subseteq X, B \subseteq Y$, we have that $\ab{e(A,B) - d(X,Y)\ab A \ab B} \leq \varepsilon \ab X \ab Y$. Szemer\'edi's regularity lemma then states that the vertex set of any graph $G$ can be partitioned as $V(G) = V_1 \cup \dots \cup V_K$ such that at least a $(1-\varepsilon)$-fraction of the pairs $(x,y) \in V(G)^2$ lie in an $\varepsilon$-quasirandom pair $(V_i,V_j)$, and moreover such that $K$ is a constant depending only on~$\varepsilon$.

Szemer\'edi's original proof shows that $K$ can be upper-bounded as a tower function\footnote{We recall that the \emph{tower function} is defined by $\twr(0)=1$ and $\twr(x+1)=2^{\twr(x)}$ for $x \geq 0$.} of $\varepsilon^{-O(1)}$, and a famous construction of Gowers \cite{MR1445389} shows that such a tower-type bound is necessary in the worst case. Unfortunately, in many applications, these tower-type losses are far too expensive to invoke. This naturally leads to the following fundamental question.
\begin{question}\label{qu:main}
	Are there natural conditions which force a graph $G$ to have regular partitions of reasonable size?
\end{question}
For example, one can ask whether bipartite graphs have smaller $\varepsilon$-regular partitions, and the answer is no: it is well-known that Gowers's \cite{MR1445389} construction can be modified to produce a bipartite graph requiring regular partitions of tower-type size.
On the other hand, it is known that if $G$ is \emph{semialgebraic of bounded description complexity}---a natural structural assumption which holds for many ``geometrically-defined'' graphs---then the bound in the regularity lemma can be improved from tower-type to polynomial \cite{MR3585030}.

\subsubsection{VC dimension}
One of the most powerful and well-studied notions of ``low complexity'' for discrete objects is the
Vapnik--Chervonenkis (VC) dimension. This notion was originally introduced by Vapnik and Chervonenkis \cite{MR288823} in the context of statistical learning theory, and over the past 50 years has become a central notion in a wide variety of areas within computer science and mathematics. Since the literature on applications of VC dimension is so vast, in what follows we will briefly mention only a handful of results related to our investigation here.

First, in computer science, VC dimension is one of the most important concepts in learning theory. Indeed, the fundamental theorem of PAC learning \cite{MR1072253} states that a binary class is learnable in Valiant's PAC model \cite{10.1145/1968.1972} if and only if the class has finite VC dimension. This foundational result lies at the heart of much of learning theory; in particular, it is remarkable that learnability---a fundamentally probabilistic and approximate notion---is characterized by a purely combinatorial notion of dimension. This result has led to a number of extensions, such as \cite{MR4537270,MR4003389}, connecting notions of learning with combinatorial measures of dimension.
Moreover, VC dimension has many applications in other areas of computer science; for example, 
in computational geometry, VC dimension is of fundamental importance because of its relation to $\varepsilon$-nets (see e.g.\ \cite[Chapter 10]{MR1899299}), which are in turn useful for many applications including range searching.

VC dimension also has many connections to and applications within mathematics. In model theory, VC dimension is intimately connected to the notion of stable models \cite{MR307903}, and gives an equivalent characterization of so-called NIP theories. In group theory, VC dimension plays an important role in the proof of the Pillay conjectures on groups definable in o-minimal theories \cite{MR2373360} (see also \cite{MR4382479}).
Finally, in combinatorics, which is the focus of the present paper, there are many different results demonstrating that discrete objects with bounded VC dimension are ``structurally simple'', and hence many difficult problems become tractable when restricted to objects of bounded VC dimension, e.g.\ \cite{MR3943496,MR4357431,1007.1670,2502.09576,MR3646875,MR4777869,2408.04165}. Many of these deep combinatorial theorems rely, at their core, on simple properties enjoyed by objects of bounded VC dimension, such as the fundamental Sauer--Shelah lemma \cite{MR307902,MR307903} and its consequences.

The formal definition of VC dimension is as follows.
Let $\F \subseteq 2^V$ be a collection of subsets of some universe $V$. The \emph{VC dimension} of $\F$ is defined to be the largest size of a set $S \subseteq V$ with the property that $S$ is \emph{shattered} by $\F$, meaning that each of the $2^{\ab S}$ possible subsets of $S$ are attained as intersections $S \cap F$ for some $F \in \F$. Given a graph $G$, its VC dimension is then defined as the VC dimension of the set system $\{N(v): v \in V(G)\} \subseteq 2^{V(G)}$, where $N(v)$ denotes the neighborhood of a vertex $v$.

\subsubsection{Question \ref{qu:main} for graphs of bounded VC dimension}
With regards to \cref{qu:main}, graphs of bounded VC dimension are particularly well-behaved with respect to the regularity lemma. Indeed, such graphs have regularity partitions with only a \emph{polynomial} number of parts. Moreover, in this regularity partition, we obtain much stronger information than simply saying that almost all pairs are $\varepsilon$-quasirandom. Indeed, let us say that a pair $(X,Y)$ is \emph{$\varepsilon$-homogeneous} if $d(X,Y) \in [0,\varepsilon] \cup [1-\varepsilon,1]$; it is easy to see that this is a strictly stronger notion than $\varepsilon$-quasirandomness (up to a polynomial change in $\varepsilon$). The \emph{ultra-strong regularity lemma} for graphs of bounded VC dimension then states the following.

\begin{theorem}\label{thm:FPS}
	Let $G$ be a graph with VC dimension $d$, and let $\varepsilon>0$. There is a partition $V(G) = V_1 \cup \dotsb \cup V_K$, where $K = (1/\varepsilon)^{O(d)}$, such that at least a $(1-\varepsilon)$-fraction of the pairs $(x,y) \in V(G)^2$ lie in $\varepsilon$-homogeneous pairs $(V_i,V_j)$.
\end{theorem}

Such a result was independently proved by Alon--Fischer--Newman \cite{MR2341924} and by Lov\'asz--Szegedy \cite{MR2815610}, and was subsequently quantitatively improved by Fox--Pach--Suk \cite{MR3943496}; the result quoted above is due to Fox--Pach--Suk, and is essentially best possible.

In what follows, we shall always view $d$, the VC dimension of $G$, as a fixed constant. Thus, \cref{thm:FPS} states that a graph of bounded VC dimension has a partition into $\poly(1/\varepsilon)$ many parts, such that all but an $\varepsilon$-fraction of the pairs of vertices lie in $\varepsilon$-homogeneous pairs.

We remark here, for future convenience, that there is a more combinatorial formulation of the property that $G$ has bounded VC dimension. Indeed, let $F$ be a bipartite graph with parts $A,B$. We say that $F$ is a \emph{bipartitely induced subgraph} of $G$ if there exist disjoint $A', B' \subseteq V(G)$ such that a pair $(a,b) \in A' \times B'$ is an edge of $G$ if and only if the corresponding pair in $A \times B$ is an edge of $F$. Equivalently, this says that $G$ contains an induced subgraph isomorphic to some supergraph of $F$, obtained by only adding new edges in the parts $A,B$, without adding or removing any edges across the bipartition.

With this definition, it is clear that $G$ has VC dimension less than $d$ if and only if $F_d$ is not a bipartitely induced subgraph of $G$, where $F_d$ is the bipartite graph with parts $A = [d]$ and $B = 2^{[d]}$, where every subset of $A$ is the neighborhood of some vertex in $B$. On the other hand, it is also easy to see that if $G$ has VC dimension greater than $d+\log d$, then for every bipartite graph $F$ with both parts of size $d$, $F$ is a bipartitely induced subgraph of $G$.

Thus, we conclude that $G$ has bounded VC dimension if and only if there exists some fixed bipartite graph $F$ which is not a bipartitely induced subgraph of $G$. This formulation will be particularly instructive and useful when we consider the corresponding concepts in hypergraphs.

Let us finally mention that a variant of \cref{thm:FPS} was obtained by Malliaris and Shelah \cite{MR3145742}. In this
version, the graph $G$ is assumed not to contain bipartitely induced copies of a specific bipartite graph (the so-called \emph{half graph}),
but the conclusion is stronger than the one in \cref{thm:FPS}, since in this case one can guarantee that all pairs in the partition
are $\varepsilon$-quasirandom. Variants of the Malliaris--Shelah theorem were also studied in the setting of Green's arithmetic regularity lemma \cite{MR2153903}, see \cite{MR3943117,MR3919562}.

\subsubsection{Hypergraph regularity and VC\texorpdfstring{$_2$}{\_2} dimension}
After the introduction of Szemer\'edi's regularity lemma, there was great interest in extending this technique to  hypergraphs. There is a natural analogue of the notion of $\varepsilon$-quasirandomness to $k$-uniform hypergraphs (henceforth $k$-graphs), and Chung \cite{MR1099803} proved (via a natural extension of Szemer\'edi's method) that every $k$-graph can be partitioned into a constant number of pieces, such that almost all $k$-tuples of vertices lie in quasirandom tuples; we discuss this notion in detail in \cref{def:weak quasirandomness,thm:chung}. Unfortunately, it was quickly discovered that this notion of hypergraph regularity is not sufficient for many applications (see the discussion before \cref{thm:hypergraph rodl}). More precisely, this notion is too weak to support a \emph{counting lemma}, which is a key ingredient of most applications of the regularity method. Thus, it remained a major research program to develop a notion of hypergraph regularity that is simultaneously strong enough to support a counting lemma and weak enough to support a regularity lemma.

This was accomplished independently in the mid-2000s, most notably by Gowers \cite{MR2373376} and by R\"odl et al.~\cite{MR2167756}. A major insight in the development of hypergraph regularity (going back to the early work of Frankl and R\"odl \cite{MR1884430}) is that, in uniformity $k \geq 3$, it is no longer sufficient to simply partition the vertices. Instead, in uniformity $k$, one must partition the vertices, the pairs of vertices, the triples of vertices, and so on, up through the $(k-1)$-tuples of vertices. This introduces a large number of conceptual, technical, and notational difficulties. As such, we defer most formal definitions and statements for now, and continue a high-level discussion during this introduction.

Our main focus will be on uniformity $3$, so let $\h$ be a $3$-graph. The hypergraph regularity lemma requires us to partition its vertices $V(\h)$, as well as the pairs $V(\h)^2$; these two partitions are called the \emph{vertex partition} and \emph{edge partition}, respectively. Each part of this latter partition can be naturally viewed as a graph, and the quasirandomness notion states very roughly that for most of these graphs $G_i$, the hyperedges of $\h$ form a ``random-like subset'' of the set of triangles in $G_i$. In order for this to be a useful notion, we also need the triangles in $G_i$ to behave in a random-like fashion, hence one needs to impose the assumption that each part $G_i$ of the edge partition is itself a quasirandom graph.

For a graph $G$, the notion of $\varepsilon$-quasirandomness is only meaningful if $\varepsilon$ is much smaller than the edge density of $G$. Hence, in the $3$-graph regularity lemma, we have both a parameter $\varepsilon$---which determines how random-like the hyperedges of $\h$ are with respect to the triangles in $G_i$---and a function $\psi:(0,1) \to (0,1)$, which determines how quasirandom the graphs $G_i$ are, as a function of their edge densities. In principle $\psi$ can be any function, but in most applications, including the counting and removal lemmas, and thus in the proofs of important consequences such as the multidimensional Szemer\'edi theorem \cite{MR2373376}, $\psi$ can be taken to be a polynomial.

The proof of the $3$-graph regularity lemma proceeds in an iterative fashion, where we refine the two partitions at every step. During a single iteration, the quasirandomness of the edge parts $G_i$ may be destroyed, hence at every step of the iteration, one must apply Szemer\'edi's regularity lemma to further refine the partition and recover the quasirandomness. Because of this, the bounds in the hypergraph regularity lemma deteriorate very quickly: even in the case of uniformity $3$, one must iteratively invoke tower-type losses, leading to a final bound of wowzer\footnote{We recall that the \emph{wowzer function} is defined by $\wow(0)=1$ and $\wow(x+1) = \twr(\wow(x))$ for $x\geq 0$.} type. Note that this happens ``almost'' regardless of what the function $\psi$ is; as long as $1/\psi$ grows slower than the tower function, the contribution of $\psi$ is completely overwhelmed by the other tower-type losses at each step. However, if $1/\psi$ grows extremely quickly, the $3$-graph regularity lemma has even worse bounds. For example, if $1/\psi$ grows like the wowzer function, then the hypergraph regularity lemma would yield bounds of $A_3$ type, where $A_3$ (the iterated wowzer function) is the next level of the Ackermann hierarchy. For more details, see e.g.\ \cite[Appendix A.1]{2404.02024}.

Recently, Moshkovitz and Shapira \cite{MR4025519} proved that such terrible bounds are in fact necessary. Concretely, in the case of uniformity $3$, they proved the existence of a $3$-graph $\h$ which has no $\varepsilon$-regular partition (for $\psi$ a polynomial) using fewer than $\wow(\log(1/\varepsilon))$ vertex parts.

Given the prior discussion, it is unsurprising that certain structural assumptions on the hypergraph $\h$ would imply better bounds. For example, it is known that, just as in the graph case, the hypergraph regularity lemma has merely polynomial bounds if $\h$ is semialgebraic of bounded description complexity. As proved by Chernikov--Starchenko \cite{MR4350155} and by Fox--Pach--Suk \cite{MR3943496}, the same conclusion also holds if $\h$ is assumed to have bounded VC dimension, for a strong notion of VC dimension that arises naturally from geometric considerations. For a different notion of VC dimension, called \emph{slicewise VC dimension}, Terry \cite{2404.01274} proved that one can obtain a regularity partition with only a double-exponential number of parts, and the authors \cite{2506.15582} recently improved her bound to single-exponential, which Terry \cite[Theorem 6.8]{2404.01293} showed to be best possible.

However, from a purely combinatorial perspective, arguably the most natural extension of VC dimension from graphs to $3$-graphs proceeds by generalizing the characterization in terms of forbidden bipartitely induced subgraphs. That is, we say that a $3$-graph $\h$ has \emph{bounded VC$_2$ dimension} if there exists a fixed tripartite $3$-graph $\V$ which does not appear as a tripartitely induced subgraph\footnote{Extending the graph terminology, we say that $\V$ is a tripartitely induced subgraph of $\h$ if there exist disjoint $A,B,C \subseteq V(\h)$, corresponding to the three vertex parts of $\V$, such that each triple $(a,b,c) \in A \times B \times C$ is an edge of $\h$ if and only if the corresponding triple is an edge of $\V$. That is, we only care about edges which go between all three parts, and we don't care if $\h$ has additional edges entirely within one part, or only between two parts.} of $\h$. This definition has its origins in model theory, where it was introduced by Shelah \cite{MR3273451}, and where it was quickly understood to be the ``correct'' higher-uniformity analogue of VC dimension (see e.g.\ \cite{MR3666349,MR3952231,MR3509704,MR4258515,MR3787371,MR2576797}), in the sense that it is a ``low-complexity'' notion which picks up on the $3$-uniform nature of the setting.

As in the graph case, this definition is equivalent to a definition in terms of shattering.
\begin{definition}
	Let $\h$ be a hypergraph and $d \geq 1$ an integer. A \emph{shattered $K_{d,d}$} in $\h$ is a pair of disjoint sets $A,B \subseteq V(\h)$ with $\ab A = \ab B = d$ with the property that, for every $S \subseteq A \times B$, there is a vertex $v \in V(\h)$ such that $(a,b,v) \in E(\h)$ if and only if $(a,b) \in S$.
	The \emph{VC$_2$ dimension} of $\h$ is then defined to be the maximum $d$ such that $\h$ contains a shattered $K_{d,d}$.
\end{definition}
Clearly, $\h$ contains a shattered $K_{d,d}$ if and only if $\h$ contains $\V_d$ as a tripartitely induced subgraph, where $\V_d$ is the tripartite $3$-graph with vertex sets $A,B,C$, where $A = B = [d]$ and $C = 2^{[d]^2}$, and where a triple $(a,b,S)$ is an edge if and only if $(a,b) \in S$. In the other direction, it is easy to see that if $\h$ has VC$_2$ dimension at least $d + \log d$, then it contains every tripartite $3$-graph with parts of size $d$ as a tripartitely induced subgraph. In this paper, we will mostly use the formulation in terms of forbidden tripartitely induced subgraphs, as this is a natural formulation from a graph-theoretic perspective, and since it is the most convenient when working with the regularity lemma.

\subsection{Hypergraph regularity under bounded VC\texorpdfstring{$_2$}{\_2} dimension}
Following the discussion above, it is natural to ask about the quantitative aspects of the $3$-graph regularity lemma if one imposes the restriction that the $3$-graph $\h$ has bounded VC$_2$ dimension; this question was first raised by Chernikov and Towsner \cite{2010.00726}, and asked explicitly by Terry \cite{MR4662634}. In particular, one could expect that the assumption of bounded VC$_2$ dimension would allow one to improve on the wowzer-type bounds, which are known to be necessary in general.

The first progress in this direction was made by Terry \cite{MR4662634}, who showed that if $\h$ has bounded VC$_2$ dimension, then the size of the {\em edge partition}\footnote{Here, and throughout the introduction, we mean the number of edge parts per pair of vertex parts when discussing the size of an edge partition.} can be taken to be only polynomial in $1/\varepsilon$. However, her proof uses the hypergraph regularity lemma as a black box, hence only provides wowzer-type upper bounds on the number of vertex parts. More recently, Terry \cite[Theorem 1.9]{2404.02024} made further progress on this question, proving that wowzer-type bounds are unavoidable in certain extreme cases. More precisely, she showed that for \emph{some} extremely fast-decreasing function $\psi_0:(0,1) \to (0,1)$, there exists a $3$-graph $\h$ with VC$_2$ dimension equal to $1$, and with the property that every $(\varepsilon,\psi_0)$-regular partition of $\h$ uses at least $\wow(\poly(1/\varepsilon))$ parts.

However, this is not the end of the story. The function $\psi_0$ used by Terry decays extremely quickly, namely $\psi_0(x) \leq 1/\wow(\poly(1/x))$. For such rapidly decaying functions, the \emph{upper} bound in the hypergraph regularity lemma is no longer of wowzer type: as discussed above, for this choice of $\psi_0$, the hypergraph regularity lemma would only deliver upper bounds of $A_3$ type. Hence, while Terry's result shows that wowzer-type bounds may be necessary even if $\h$ has bounded VC$_2$ dimension, such a lower bound no longer matches the upper bound.

Additionally, as discussed previously, most of the important applications of the hypergraph regularity lemma do not use such rapidly decaying $\psi$. Instead, applications usually only use $\psi$ which are polynomial functions, and it is in such cases that the hypergraph regularity lemma yields wowzer-type upper bounds. As such, it is natural to ask (as was done explicitly in \cite{2404.02024}) whether the wowzer-type bound can be improved if $\h$ has bounded VC$_2$ dimension, and if one restricts $\psi$ to be a ``sensible'' function. Our main result in this paper shows that this is indeed the case.

\begin{theorem}[Informal statement; see \cref{thm:main UB formal}]\label{thm:main UB informal}
Let $\psi:(0,1) \to (0,1)$ be a polynomial function and let $\varepsilon>0$. If $\h$ is a $3$-graph with bounded VC$_2$ dimension, then $\h$ has an $(\varepsilon,\psi)$-regular partition with at most $\twr(\twr(\poly(1/\varepsilon)))$ vertex parts.
\end{theorem}

A few remarks about \cref{thm:main UB informal} are in order. First, we note that the double-tower-type bound in \cref{thm:main UB informal} is much better than the wowzer-type bound one has in general, for the same reason that $2^{2^x} \ll \twr(x)$. Just as the double-exponential function $2^{2^x}$ is of ``exponential type'' rather than of tower type, so too the bound in \cref{thm:main UB informal} is of tower type, rather than wowzer type. Second, we remark that the assumption that $\psi$ is a polynomial function is not really important; the same result holds so long as $1/\psi$ grows much slower than a tower function, e.g.\ $\psi(x) = 1/2^{2^{2^{2^x}}}$ also works. Essentially, if $\psi$ is such that the usual hypergraph regularity lemma yields wowzer-type bounds, our \cref{thm:main UB informal} yields a double-tower bounds under the added assumption of bounded VC$_2$ dimension. Third, we remark that our proof of \cref{thm:main UB informal} 
also gives a partition where the number of edge parts is tower type, but this is substantially less interesting: as discussed above, the result of Terry \cite[Theorem 3.1]{MR4662634} shows that we may take the number of edge parts to be polynomial, and moreover this is best possible \cite{2404.02030}. 
For many more results and questions on the sizes of the edge partition, see \cite{2404.02030}, as well as \cref{conj:graph partition} in the concluding remarks. Fourth, we recall that Terry \cite{2404.02024} proved a wowzer-type lower bound if $\psi$ is an arbitrary function, and \cref{thm:main UB informal} gives a tower-type upper bound for ``reasonable'' $\psi$; it would be interesting to determine which exact growth rate for $\psi$ determines the transition from tower-type bounds to wowzer-type bounds for regularity partitions of hypergraphs with bounded VC$_2$ dimension.

Fifth, we remark that the structure given by \cref{thm:main UB informal} is stronger than that of a usual regularity partition, in the same way that \cref{thm:FPS} is both qualitatively and quantitatively stronger than Szemer\'edi's regularity lemma. Namely, not only do we ensure that the hyperedges in $\h$ form a quasirandom subset of the set of triangles of most graphs $G_i$ in the edge partition, but rather, we show that the hyperedges comprise either almost all or almost none of these triangles. This is a stronger statement, in the same way that $\varepsilon$-homogeneity of a pair of vertex sets is a stronger statement than $\varepsilon$-quasirandomness. It was already proved by Chernikov--Towsner \cite{2010.00726} and Terry--Wolf \cite{2111.01737} that $3$-graphs of bounded VC$_2$ dimension admit such a decomposition, but their techniques do not give the quantitative improvement on the number of vertex parts over the usual hypergraph regularity lemma\footnote{In fact, the work of Chernikov--Towsner \cite{2010.00726} uses infinitary techniques that give no bound at all; the work of Terry--Wolf \cite{2111.01737} is quantitative, but uses the hypergraph regularity lemma as a black box and hence does not improve over its quantitative dependencies.}.

Finally, we recall that $\h$ having bounded VC$_2$ dimension is equivalent to $\h$ forbidding some fixed tripartite $3$-graph $\V$ as a tripartitely induced subgraph. In fact, the proof of \cref{thm:main UB informal} does not require that there be \emph{zero} tripartitely induced copies of $\V$; instead, it suffices for there to be \emph{few} such copies, i.e.\ $o(\ab \h^{\ab \V})$ such copies. This is quite commonplace in the study of regularity and VC dimension. For example, the regularity lemma for graphs of bounded VC dimension of \cite{MR2341924} is explicitly stated under the assumption of few bipartitely induced copies, and moreover, it is proved in \cite{2111.01737} that a $3$-graph $\h$ admits a homogeneous partition, in the sense above, if and only if $\h$ has few tripartitely induced copies of a fixed tripartite $\V$. That is, the ``correct'' notion of bounded VC$_2$ dimension is, in some sense, having few copies rather than having zero copies. However, we opted not to state \cref{thm:main UB informal} in this way in order to make the statement cleaner.

Given that \cref{thm:main UB informal} improves the upper bound from a wowzer function to a tower-type function, it is natural to ask whether it can be improved even further. As remarked above, for two stronger notions of bounded VC dimension for hypergraphs, Chernikov--Starchenko \cite{MR4350155}, Fox--Pach--Suk \cite{MR3943496}, Terry \cite{2404.01274}, and the authors \cite{2506.15582} proved that this is indeed the case, and that one can obtain, respectively, polynomial and single-exponential bounds. 
However, in the case of
VC$_2$ dimension, this is impossible: tower-type bounds are the truth, as proved by Terry \cite{2404.01293}.

\begin{proposition}[Terry \cite{2404.01293}]\label{thm:main LB informal}
	For every $\varepsilon>0$ and every $\psi:(0,1) \to (0,1)$ with $\psi(x) \leq x^C$ for some absolute constant $C$, there exist infinitely many $3$-graphs $\h$ with the following properties. $\h$ has VC$_2$ dimension at most $1$, and every $(\varepsilon,\psi)$-regular partition of $\h$ has at least $\twr(\poly(1/\varepsilon))$ vertex parts.
\end{proposition}
In fact, Terry proved something somewhat stronger, namely that every \emph{weak} $\varepsilon$-regular partition of $\h$ has at least $\twr(\poly(1/\varepsilon))$ parts (see \cref{prop:terry LB}). We discuss this in more detail below, but briefly, a weak $\varepsilon$-regular partition is the structure given by Chung's \cite{MR1099803} early work on hypergraph regularity. As its name suggests, this is a strictly weaker structure than that given by the full hypergraph regularity lemma, hence a lower bound on the number of parts in such a partition immediately implies the same lower bound on the number of parts in a full-fledged hypergraph regularity partition. Although \cref{thm:main LB informal} follows easily from results of Terry, it has apparently not been explicitly stated in this form before, so we include the details of the derivation in \cref{sec:LB proof}.

By contrast, the proof of \cref{thm:main UB informal} requires substantial new ideas, and so we give a high-level outline of it in the next subsection.

\subsection{Outline of the proof of Theorem \ref{thm:main UB informal}}\label{sec:proof sketch}
\subsubsection{Difficulties with extending the proof of Theorem \ref{thm:FPS}}

Given that \cref{thm:main UB informal} is a uniformity-$3$ extension of \cref{thm:FPS}, it is natural to attempt to mimic the proof of \cref{thm:FPS}. Unfortunately, this appears to be impossible, because of fundamental differences between the notions of VC dimension and VC$_2$ dimension. In what follows we elaborate on this issue.

We first briefly discuss the proof of \cref{thm:FPS}, as well as why it does not work in uniformity $3$. Although this is not directly relevant to what follows---after all, this is an approach that fails---we believe it is still instructive, as it demonstrates some of the key differences between graphs and $3$-graphs that make proving results like \cref{thm:main UB informal} so much harder than results like \cref{thm:FPS}. The impatient reader can safely skip to \cref{sec:our approach sketch}, where we explain the approach we actually use in the proof of \cref{thm:main UB informal}.

As mentioned above, results like \cref{thm:FPS} were proved by Lov\'asz--Szegedy \cite{MR2815610}, Alon--Fischer--Newman \cite{MR2341924}, and Fox--Pach--Suk \cite{MR3943496}. These three proofs have certain differences, but a key feature shared by all three is the use of the fact that set systems with bounded VC dimension cannot contain many sets with large pairwise symmetric difference. The strongest version of this fact is
Haussler's packing lemma \cite{MR1313896}, which is one of the central tools in the theory of VC dimension.
\begin{lemma}[Haussler \cite{MR1313896}]\label{lem:haussler}
	Let $\F \subseteq 2^{[n]}$ be a set system of VC dimension $d$. If $\ab{S \symmd T} \geq \varepsilon n$ for all distinct $S,T \in \F$, then $\ab \F =O_d(\varepsilon^{-d})$.
\end{lemma}
Using this lemma, Fox, Pach, and Suk proved \cref{thm:FPS} as follows. Let $G$ be a graph with bounded VC dimension, and let us suppose for simplicity that $G$ is bipartite\footnote{This proof actually works equally well whether or not $G$ is bipartite. However, it will be more convenient to work with bipartite graphs in what follows.} with parts $A,B$, where $\ab A = \ab B = n$. Pick out a maximal collection $a_1,\dots,a_k \in A$ with the property that $\ab{N(a_i) \symmd N(a_j)} \geq \varepsilon n$. By \cref{lem:haussler}, we have that $k \leq \poly(1/\varepsilon)$, since the set system $\{N(a_i)\}_{1\leq i \leq k} \subseteq 2^B$ has bounded VC dimension. Additionally, by maximality, we find that for every $v \in A$, we have that $\ab{N(v) \symmd N(a_i)} < \varepsilon n$ for some $i \in [k]$. Therefore, we may partition $A$ into $A_1,\dots,A_k$, where $A_i$ comprises only vertices $v \in A$ with $\ab{N(v) \symmd N(a_i)} < \varepsilon n$. The triangle inequality then implies that $\ab{N(v) \symmd N(v')} <2\varepsilon n$ for all $v,v' \in A_i$. By interchanging the roles of $A$ and $B$, we can also partition $B$ into $B_1,\dots,B_\ell$, where $\ell \leq \poly(1/\varepsilon)$, such that $\ab{N(u) \symmd N(u')} <2\varepsilon n$ for all $u,u' \in B_i$.

In other words, we have partitioned $V(G)$ into $\poly(1/\varepsilon)$ parts, such that every pair of vertices in the same part have the same neighborhood, up to an error of $O(\varepsilon n)$. At this point, a short double-counting argument (see \cite[Lemma 2.3]{MR3943496}) shows that almost all pairs $(A_i,B_j)$ are $O(\varepsilon)$-homogeneous.

Let us attempt to extend such a proof to the uniformity-$3$ setting. So let $\h$ be a tripartite $3$-graph with parts $A,B,C$, where $\ab A = \ab B = \ab C=n$, and suppose that $\h$ has bounded VC$_2$ dimension. For $c \in C$, let us denote by $N(c) \subseteq A \times B$ the \emph{link} of $c$, i.e.\ the set of pairs $(a,b)$ such that $(a,b,c) \in E(\h)$. Following the proof above, it is natural to select a maximal collection $c_1,\dots,c_k \in C$ with $\ab{N(c_i) \symmd N(c_j)} \geq \varepsilon n^2$, and to partition $C$ accordingly.

However, one immediately encounters a fundamental problem, which is that there is no analogue of Haussler's packing lemma for VC$_2$ dimension. Concretely, in the attempt above, we cannot get \emph{any} bound on $k$ only in terms of $\varepsilon$, hence we cannot use this technique to partition $V(\h)$ into $O_\varepsilon(1)$ parts. To see this, we define a tripartite $3$-graph $\h$ with parts $A,B,C$, where $\ab A = \ab B = \ab C = n$, as follows. For every $c \in C$, we pick uniformly random sets $X_c \subseteq A, Y_c \subseteq B$, making these choices independently. We then define
\begin{equation}\label{eq:complete bipartite links}
	E(\h) \coloneqq \{(x,y,c) : c \in C, x \in X_c, y \in Y_c\}.
\end{equation}
In other words, the link of every vertex $c \in C$ is precisely the complete bipartite graph between $X_c$ and $Y_c$. It is easy to see that, with probability $1$, $\h$ does not contain a tripartitely induced copy of $K_{2,2,2}^{(3)}\setminus e$, hence $\h$ has bounded VC$_2$ dimension (in fact, one can check that $\h$ has no shattered $K_{2,2}$, so its VC$_2$ dimension equals $1$). On the other hand, with high probability, we have that $\ab{X_c \symmd X_{c'}}, \ab{Y_c \symmd Y_{c'}} = \Omega(n)$, for all distinct $c,c' \in C$. Therefore $\ab{N(c) \symmd N(c')} = \Omega(n^2)$ for all $c,c' \in C$, and the approach discussed above would entail partitioning $C$ into singletons.

The fact that there is no analogue of Haussler's packing lemma for the VC$_2$ dimension demonstrates a fundamental difference between the notions of VC and VC$_2$ dimension. As remarked earlier, Fox, Pach, and Suk \cite{MR3943496} were able to use the same approach to prove a polynomial regularity lemma for hypergraphs which satisfy a much stronger form of bounded VC dimension, precisely because this stronger assumption does support such a packing lemma. However, if we only assume bounded VC$_2$ dimension, we need to use an alternative approach.

What we infer from the above discussion is that before we attempt to tackle the $3$-graph case, we should first prove \cref{thm:FPS} without relying on \cref{lem:haussler} (Haussler's packing lemma), hoping that such a proof can then be extended to the setting of $3$-graphs. This is precisely the content of the next subsection.

\subsubsection{A Haussler-free proof of Theorem \ref{thm:FPS}}\label{subsec:hausslerfree}

Our goal in this subsection is to sketch the proof of the following result.

\begin{proposition}\label{prop:graph double exp}
	Let $G$ be a graph with bounded VC dimension, and let $\varepsilon>0$. There is a partition of $V(G)$ into $K \leq 2^{2^{\poly(1/\varepsilon)}}$ parts such that a $(1-\varepsilon)$-fraction of the pairs $(x,y) \in V(G)^2$ lie in $\varepsilon$-homogeneous pairs of parts.
\end{proposition}
Note that this result is much weaker than \cref{thm:FPS}, as the number of parts is double-exponential, rather than the optimal polynomial dependence in \cref{thm:FPS}. However, the proof strategy we now discuss has the advantage of being robust enough that it can be extended to the setting of $3$-graphs with bounded VC$_2$ dimension (with some significant technical difficulties). Furthermore, this double-exponential loss will turn out to be negligible in the $3$-uniform setting.

For simplicity, let us first add the assumption that $G$ is bipartite with parts $A,B$. The key new idea in the proof is to not try to immediately construct a partition of $V(G)$. Instead, we begin by finding a weaker structure, called a \emph{cylinder partition}.
\begin{definition}
	Let $G$ be a bipartite graph with parts $A,B$. A \emph{vertex cylinder} is simply a product set $A' \times B' \subseteq A \times B$, where $A' \subseteq A, B' \subseteq B$. A cylinder $A' \times B'$ is $\eta$-quasirandom if the pair $(A',B')$ is $\eta$-quasirandom in $G$. A \emph{(vertex) cylinder partition} of $G$ is a partition of $A \times B$ into vertex cylinders $A_1 \times B_1,\dots,A_k \times B_k$.
\end{definition}
The \emph{cylinder regularity lemma} of Duke, Lefmann, and R\"odl\footnote{In fact, for bipartite graphs, the cylinder regularity lemma essentially goes back to the original work of Szemer\'edi \cite{MR369312} on arithmetic progressions in dense sets of integers, and precedes the ``usual'' regularity lemma by several years.} \cite{MR1333857} states that every bipartite graph has a cylinder partition into $k \leq 2^{\poly(1/\eta)}$ parts, such that a $(1-\eta)$-fraction of the pairs $(x,y) \in A \times B$ lie in an $\eta$-quasirandom cylinder $A_i \times B_i$. Let us fix some parameter $\eta = \poly(\varepsilon)>0$ depending polynomially on $\varepsilon$, and fix some $\eta$-quasirandom cylinder partition of $G$ with $k \leq 2^{\poly(1/\eta)} = 2^{\poly(1/\varepsilon)}$ parts.

Now consider some $\eta$-quasirandom cylinder $A_i \times B_i$. The key observation is that the edge density $d(A_i,B_i)$ must be very close to $0$ or $1$. Indeed, suppose that $\delta \leq d(A_i,B_i) \leq 1-\delta$, for some $\delta>0$. The \emph{induced counting lemma} (e.g.\ \cite[Lemma 3.2]{MR1804820}) implies that if $\eta< \delta^{d^2}$, then the pair $(A_i,B_i)$ contains a bipartitely induced copy of every bipartite graph with both parts of size $d$. However, since we assumed that $G$ has bounded VC dimension, it does not contain any bipartitely induced copy of some fixed bipartite graph, a contradiction. Thus, we conclude that every $\eta$-quasirandom cylinder in our partition must be $\delta$-homogeneous.

We are almost done; all that remains is to convert this cylinder partition into an honest partition of $V(G)$. There is an obvious way of doing this, namely to take the \emph{Venn diagram partition}, namely to partition $A$ according to the common refinement of all the sets $A_i$, and similarly for $B$. By doing this, we may exponentiate the number of parts, so we end up partitioning $V(G)$ into $K \leq 2^{O(k)} \leq 2^{2^{\poly(1/\varepsilon)}}$ parts.

If all we knew about the cylinder partition was that most cylinders are quasirandom, this structure would be completely destroyed in the Venn diagram partition; it is for this reason that one cannot simply prove Szemer\'edi's regularity lemma by applying the cylinder regularity lemma and then taking the Venn diagram partition. However, we have more information, namely that most cylinders in the partition are $\delta$-homogeneous. A simple application of Markov's inequality now shows that if we pick $\delta = \varepsilon^2$, then this homogeneity property \emph{is} preserved when we pass to the Venn diagram partition. Concretely, if a $(1-\varepsilon^2)$-fraction of the pairs of vertices lie in $\varepsilon^2$-homogeneous cylinders, then in \emph{any} refinement of the partition, a $(1-\varepsilon)$-fraction of the pairs of vertices lie in $\varepsilon$-homogeneous pairs of parts. To conclude the proof, we note that we only take $\delta$ to depend polynomially on $\varepsilon$ and $\eta$ to depend polynomially on $\delta$ (with some dependence on the VC dimension of $G$, which we treat as a fixed constant), hence we have $\eta = \poly(\varepsilon)$ as promised.

Finally, let us discuss how to lift the assumption that $G$ is bipartite. The cylinder regularity lemma actually holds for $t$-partite graphs, for any $t\geq 2$. Namely, if $G$ is $t$-partite with parts $X_1,\dots,X_t$, a cylinder partition is a partition of the product set $X_1 \dtimes X_t$ into cylinders $X_1' \dtimes X_t'$. Duke, Lefmann, and R\"odl \cite{MR1333857} proved that one can find such a cylinder partition with $2^{\poly(t/\eta)}$ parts, such that a $(1-\eta)$-fraction of tuples $(x_1,\dots,x_t)$ lie in cylinders $X_1'\dtimes X_t'$ such that all pairs $(X_i',X_j')$ are $\eta$-quasirandom. Now, given a general graph $G$, we first begin by arbitrarily partitioning it into $t=1/\varepsilon$ parts of equal sizes, and treating it as a $t$-partite graph (i.e.\ forgetting the edges inside a part). We now run the argument above, noting that since $t = \poly(1/\varepsilon)$ and since the cylinder regularity lemma has a bound of the form $2^{\poly(t/\eta)}$, we still end up with a partition of $V(G)$ into $2^{2^{\poly(1/\varepsilon)}}$ parts, such that almost all pairs are $\varepsilon$-homogeneous. Finally, we note that by the choice of $t=1/\varepsilon$, the same conclusion holds (up to replacing $\varepsilon$ by $2\varepsilon$) even if we bring back the edges we ignored when we artificially made $G$ $t$-partite. This concludes the proof sketch of \cref{prop:graph double exp}.

\subsubsection{Proof sketch of Theorem \ref{thm:main UB informal}}\label{sec:our approach sketch}

In order to use the strategy we described in the previous subsection in uniformity $3$, we need a few ingredients. First, we need a $3$-graph analogue of the cylinder regularity lemma, which is the main technical contribution of this paper (see \cref{thm:hyper cylinder} for the statement). Crucially, this cylinder regularity lemma yields a partition whose size is only of tower type, much better than the wowzer-type bounds needed to obtain a full regularity partition; this is analogous to the situation in graphs, where the full regularity lemma requires tower-type bounds, but the cylinder regularity lemma requires only exponential bounds. Unsurprisingly, even the notion of a cylinder partition becomes substantially more complicated when we move to hypergraphs. We defer a detailed discussion of these definitions to \cref{sec:cylinder defs}.

Let us now fix such a cylinder partition of $\h$, as given by the cylinder regularity lemma. The next thing we need is an induced counting lemma, in order to conclude that all the quasirandom cylinders are actually very sparse or very dense. As discussed earlier, the main issue in the development of the hypergraph regularity method is finding notions of quasirandomness that can yield such counting lemmas, while also supporting a corresponding regularity lemma. For our application, it is crucial that the counting lemma only requires polynomial bounds, so that we may set $\eta= \poly(\varepsilon)$ as we did in the graph case (at the end of the penultimate paragraph of \cref{subsec:hausslerfree}). While it is plausible that each of the notions of hypergraph regularity does indeed only require polynomial losses in the counting lemma, as far as we are aware, no one has proved this for the regularity notion of R\"odl et al.~\cite{MR2167756}. On the other hand, Gowers's \cite{MR2373376} approach to the hypergraph regularity lemma does give a counting lemma with polynomial bounds, hence we use Gowers's notion throughout. However, we remark that the final outcome is essentially agnostic\footnote{It is known that all the notions of hypergraph regularity are qualitatively equivalent (see e.g.\ \cite{MR4816408}). However, the proofs of these equivalences require applications of the hypergraph regularity lemma itself, and hence are not quantitatively effective. Since the main thrust of \cref{thm:main UB informal} is the quantitative improvement, we cannot freely pass between different regularity notions.} about which notion of regularity is used, because our final partition actually has the stronger $\varepsilon$-homogeneity property, which is easily seen to be stronger than each of the notions of $\varepsilon$-regularity.

We now have a cylinder partition of $\h$ in which almost all cylinders are very sparse or very dense. As before, we now pass to the Venn diagram partition, obtaining a genuine vertex partition and edge partition of $\h$. Unlike in the graph case, this step is essentially ``for free'': although we exponentiate the number of parts, this loss is negligible given that the number of parts in the cylinder partition is already of tower type. As before, this step maintains the $\varepsilon$-homogeneity property, that is, almost all triples of vertices now lie in parts that are very sparse or very dense.

We are almost done, except that in a hypergraph regularity partition, we require the graphs arising from the edge partition to themselves be quasirandom. This was the case in the cylinder partition, but this property may have been destroyed when passing to the Venn diagram partition. Hence, we now need to apply Szemer\'edi's regularity lemma in order to further refine the vertex partition, and to recover the quasirandomness of the graphs. This final step costs us another tower function, which is why the final bound is of double-tower type. However, we stress that we only apply the graph regularity lemma \emph{once}, whereas it needs to be applied iteratively in the proof of the usual hypergraph regularity lemma\footnote{That said, the cylinder partition we use {\em is} required to have its underlying graph be quasirandom, in order to apply the induced counting lemma and conclude that almost all cylinders are very sparse or very dense. As such, in the proof of our cylinder regularity lemma for hypergraphs, we are required to regularize the graph at each step of the iteration. However, as we are working with cylinder partitions, we may apply the cylinder regularity lemma for graphs, rather than Szemer\'edi's regularity lemma, at every step of the iteration, and thus only pay an exponential loss at every step; this is why our cylinder regularity lemma ends up incurring tower-type bounds.}. This difference is what ensures that we end up with a final bound of tower type, rather than wowzer type. To summarize, in the proof of the hypergraph regularity lemma, one iteratively applies Szemer\'edi's regularity lemma, leading to wowzer-type bounds. In our proof, we iteratively apply the {cylinder} regularity lemma for graphs (incurring tower-type dependencies), and then apply Szemer\'edi's regularity lemma only once (losing one more tower).

\subsubsection{Further directions}
The cylinder regularity lemma for graphs is extremely useful and has numerous applications, so we naturally expect our cylinder regularity lemma for $3$-graphs to have many applications as well. We discuss two such applications in \cref{sec:cylinder apps}. In particular, one of these applications is a $3$-graph extension of a famous theorem of R\"odl \cite{MR0837962}, closely related to the Erd\H os--Hajnal conjecture (see \cite{2403.08303}), about sparse or dense sets in graphs forbidding a fixed induced subgraph. As observed already in R\"odl's original paper \cite{MR0837962}, the most natural extension of this theorem to $3$-graphs is false; nonetheless, our cylinder regularity lemma naturally implies that a subtler statement\footnote{In fact, the same qualitative statement can be deduced from the full hypergraph regularity lemma, but with worse quantitative bounds.} is true. As our cylinder regularity lemma invokes tower-type bounds, we obtain tower-type bounds in our applications, and we would be very interested to see if these bounds could be improved; see \cref{qu:subtower for applications} for details.

Finally, we remark that there are natural generalizations of all questions studied in this paper to hypergraphs of higher uniformity. We plan to address such questions in forthcoming work.

\paragraph{Organization}
The rest of this paper is organized as follows. In \cref{sec:definitions}, we define the relevant notions of hypergraph regularity and give formal statements of our main theorems. We develop the cylinder regularity lemma for hypergraphs in \cref{sec:cylinder}. We use the cylinder regularity lemma, and the proof strategy outlined above, to prove \cref{thm:main UB informal} in \cref{sec:UB proof}. In \cref{sec:LB proof}, we prove the lower bound, \cref{thm:main LB informal}, from the results of Terry \cite{2404.01293,2404.02024}. We conclude in \cref{sec:conclusion} with some final remarks and open problems.

\section{Formal definitions and statements}\label{sec:definitions}
\subsection{Definitions regarding hypergraph regularity}
As mentioned in the introduction, we use Gowers's notion of hypergraph regularity \cite{MR2195580,MR2373376}. In order to make this paper as accessible as possible to non-experts, we follow the notation of \cite{MR2195580} whenever possible. We will give minimal explanation or justification for most of these definitions, and defer to \cite{MR2195580}; that paper is an excellent resource for those wishing to learn about hypergraph regularity, and includes detailed explanations of the various notions we use.

We begin by defining what it means for a pair of vertex sets in a graph to be quasirandom. We recall that $d_G(X,Y)$ denotes the edge density between two vertex sets in a graph $G$; we now start including the $G$ subscript (which we omitted before), in order to disambiguate densities of different objects.
\begin{definition}[{\cite[Definitions 3.2 and 3.3]{MR2195580}}]\label{def:graph quasirandomness}
	Let $X,Y$ be finite sets, and let $f: X \times Y \to [-1,1]$ be a function. For a real number $\alpha>0$, we say that $f$ is \emph{$\alpha$-quasirandom} if
	\[
		\sum_{x,x' \in X} \sum_{y,y' \in Y} f(x,y) f(x',y) f(x,y') f(x',y') \leq \alpha \ab X^2 \ab Y^2.
	\]
	Let $G$ be a graph, and let $X,Y \subseteq V(G)$. We say that the pair $(X,Y)$ is \emph{$\alpha$-quasirandom} if the function $f:X \times Y \to [-1,1]$ defined by
	\[
		f(x,y) \coloneqq \1_{xy \in E(G)} - d_G(X,Y)
	\]
	is $\alpha$-quasirandom.

	If $G$ is a $t$-partite graph with parts $X_1,\dots,X_t$, we say that $G$ is \emph{$\alpha$-quasirandom} if each of the pairs $(X_i,X_j)$ for $1 \leq i<j\leq t$ is $\alpha$-quasirandom.
\end{definition}
This is a different definition than the definition of $\alpha$-quasirandomness given in the introduction, but it is not hard to show, via a few applications of the Cauchy--Schwarz inequality, that these two definitions are equivalent\footnote{This equivalence is (essentially) one of the many equivalences proven in the seminal work of Chung, Graham, and Wilson \cite{MR1054011}, namely that quasirandomness can be characterized by $C_4$ counts. Indeed, the sum appearing in \cref{def:graph quasirandomness} is, up to lower-order terms, equal to the count of 4-cycles in $X \times Y$, minus the expected number in a random graph of the same density.} up to a polynomial change in the value of $\alpha$. See \cite[Theorem 3.1]{MR2195580} for details.

The definition of quasirandomness for $3$-graphs will be similar, but as discussed in the introduction, the quasirandomness notion has to do with the way in which the edges of a $3$-graph lie within the triangles of a graph on the same vertex set. To discuss this structure, we make the following definition.
\begin{definition}[{\cite[Definition 6.1]{MR2195580}}]
	Given a graph $G$, let $\Delta(G)$ denote the set of triangles in $G$. A \emph{chain} is a pair $(G,\h)$, where $G$ is a graph and $\h$ is a $3$-graph, such that $G$ and $\h$ have the same vertices, and $E(\h) \subseteq \Delta(G)$.
	
	Similarly, if $G$ is a tripartite graph with parts $X,Y,Z$, we say that a function $f:X \times Y \times Z \to \R$ is \emph{supported on $\Delta(G)$} if $f(x,y,z)\neq 0$ only if $xyz \in \Delta(G)$.

	Given $t \geq 3$, a \emph{$t$-partite chain} is a chain in which $G$ is a $t$-partite graph, $\h$ is a $t$-partite $3$-graph, both with respect to the same partition $V(G)=V(\h) = X_1 \sqcup \dots \sqcup X_t$.

	If $(G,\h)$ is a chain, the \emph{relative density} of $\h$ in $G$ is defined as
	\[
		d(\h\mid  G) \coloneqq \frac{\ab{E(\h)}}{\ab{\Delta(G)}},
	\]
	with the convention that $0/0=0$.
\end{definition}
We are now ready to define quasirandomness\footnote{We remark that, similarly to how \cref{def:graph quasirandomness} captures the count of $C_4=K_{2,2}$ in a bipartite graph, \cref{def:quasirandom chain} captures the count of the octahedron hypergraph $K_{2,2,2}^{(3)}$.} for a $3$-graph relative to a graph.
\begin{definition}[{\cite[Definition 6.3]{MR2195580}}]\label{def:quasirandom chain}
	Let $G$ be a tripartite graph with parts $X,Y,Z$, and let $f:X \times Y \times Z \to [-1,1]$ be a function supported on $\Delta(G)$. Given $\eta>0$, we say that $f$ is \emph{$\eta$-quasirandom relative to $G$} if
	\begin{multline*}
		\sum_{\substack{x,x' \in X\\y,y' \in Y\\z,z' \in Z}} f(x,y,z) f(x,y,z') f(x,y',z) f(x,y',z') f(x',y,z) f(x',y,z') f(x',y',z) f(x',y',z') \\
		\leq \eta [d_G(X,Y) d_G(X,Z) d_G(Y,Z)]^4 \ab X^2 \ab Y^2 \ab Z^2.
	\end{multline*}
	If $(G,\h)$ is a tripartite chain, we say that $\h$ is \emph{$\eta$-quasirandom relative to $G$} if the function $f:X \times Y \times Z \to [-1,1]$ defined by
	\[
		f(x,y,z) \coloneqq (\1_{xyz \in E(\h)} - d(\h\mid G))\1_{xyz \in \Delta(G)}
	\]
	is $\eta$-quasirandom.

	Finally, if $(G,\h)$ is a $t$-partite chain for some $t \geq 3$, then we say that $\h$ is \emph{$\eta$-quasirandom relative to $G$} if this holds for each of the $\binom t3$ tripartite subchains.
\end{definition}
As discussed in the introduction, in our regularity decompositions, we require not only that $\h$ is quasirandom with respect to the graphs in the edge partition, but also that these graphs are themselves highly quasirandom. This idea is captured in the following definition.
\begin{definition}[{\cite[Definition 8.5]{MR2195580}}]\label{def:product density}
	Let $(G,\h)$ be a $t$-partite chain with parts $X_1,\dots,X_t$. The \emph{product density} of $G$ is defined as
	\[
		\delta(G) \coloneqq \prod_{1 \leq i <j \leq t} d_G(X_i,X_j).
	\]
	Given a number $\eta>0$ and a function $\psi:(0,1) \to (0,1)$, we say that the chain $(G,\h)$ is \emph{$(\eta,\psi)$-quasirandom} if $\h$ is $\eta$-quasirandom relative to $G$, and if $G$ is $\psi(\delta(G))$-quasirandom.
\end{definition}
Finally, in order to state the hypergraph regularity lemma, we need to define chain partitions of hypergraphs.
\begin{definition}[{\cite[Definition 8.7]{MR2195580}}]\label{def:chain partitions}
	Let $\h$ be a $3$-graph. A \emph{chain partition} of $\h$ is a pair $\Q=(\Q_V,\Q_E)$, where $\Q_V$ is a partition of $V(\h)$ into parts, and $\Q_E$ consists, for all distinct $Y_i,Y_j \in \Q_V$, of a partition $\Q_E(Y_i \times Y_j)$ of the complete bipartite graph $Y_i \times Y_j$ into subgraphs.

	For a vertex $x \in V(\h)$, we denote by $\Q_V[x]$ the unique part of $\Q_V$ containing $x$. For distinct $x_1,x_2 \in V(\h)$, we denote by $\Q_E[x_1,x_2]$ the unique part of $\Q_E(\Q_V[x_1]\times \Q_V[x_2])$ containing the edge $(x_1,x_2)$. Note that this notation is only well-defined in case $x_1$ and $x_2$ lie in distinct parts of $\Q_V$.

	Given disjoint sets $Y_1,Y_2,Y_3 \subseteq V(\h)$ and $Z_{12} \subseteq Y_1 \times Y_2, Z_{13} \subseteq Y_1 \times Y_3, Z_{23} \subseteq Y_2 \times Y_3$, we denote by $\h[Y_1 \cup Y_2 \cup Y_3; Z_{12} \cup Z_{13} \cup Z_{23}]$ the chain $(G,\h')$, where $V(G)=Y_1 \cup Y_2 \cup Y_3$, $E(G)=Z_{12} \cup Z_{13} \cup Z_{23}$, and where $\h'$ is the tripartite subgraph of $\h$ with $V(\h')=Y_1 \cup Y_2 \cup Y_3$ and $E(\h') = E(\h) \cap \Delta(G)$.

	Given a chain partition $\Q$ and vertices $x_1,x_2,x_3 \in V(\h)$, we denote by $\h[\Q;x_1,x_2,x_3]$ the chain $\h[\Q_V[x_1] \cup \Q_V[x_2] \cup \Q_V[x_3]; \Q_E[x_1,x_2] \cup \Q_E[x_1,x_3] \cup \Q_E[x_2,x_3]]$. Again, this notation is only well-defined if $x_1,x_2,x_3$ lie in different parts of $\Q_V$.

	Finally, a chain partition $\Q=(\Q_V,\Q_E)$ of $\h$ is said to be \emph{$(\eta,\psi)$-quasirandom} if, for at least a $(1-\eta)$-fraction of the triples $(x_1,x_2,x_3) \in V(\h)^3$, the chain $\h[\Q;x_1,x_2,x_3]$ is $(\eta,\psi)$-quasirandom as in \cref{def:product density}.
\end{definition}
Recall that the notation $\h[\Q;x_1,x_2,x_3]$ is only defined for triples $(x_1,x_2,x_3)$ lying in three different parts of $\Q$. Thus, when we say that the quasirandomness holds for at least a $(1-\eta)$-fraction of triples, we mean that this notation is defined, and is $(\eta,\psi)$-quasirandom, for at least $(1-\eta)\ab{V(\h)}^3$ triples.

We are finally ready to state the regularity lemma for $3$-graphs.
\begin{theorem}[{\cite[Theorem 8.10]{MR2195580}}]\label{thm:gowers regularity}
	For all $\eta>0$ and $\psi:(0,1) \to (0,1)$, there exists a constant $K=K(\eta,\psi)$ such that the following holds. For every $3$-graph $\h$, there is a chain partition $\Q=(\Q_V,\Q_E)$ which is $(\eta,\psi)$-quasirandom.
	
	Additionally, $\ab{\Q_V} \leq K$ and $\ab{\Q_E(Y_i \times Y_j)} \leq K$ for all distinct $Y_i,Y_j \in \Q_V$. Moreover, $K$ is bounded by $\poly(1/\eta)$ iterations of the function $\twr(\poly(1/\psi(\eta)))$.
\end{theorem}
The explicit bound on $K$ is discussed but not proved in \cite{MR2195580}, and a formal proof is given in \cite[Appendix A]{2404.02024}.

To conclude this section, we briefly discuss weak regularity for hypergraphs, as this will be relevant in the proof of \cref{thm:main LB informal}. The following definition is a direct extension of the definition of graph quasirandomness discussed in the introduction. It is not quite the definition introduced by Chung \cite{MR1099803}, but it is easily seen to be equivalent up to a polynomial change in $\eta$.
\begin{definition}\label{def:weak quasirandomness}
	Let $\h$ be a $3$-graph, and let $X_1,X_2,X_3 \subseteq V(\h)$. We let $e_\h(X_1,X_2,X_3)$ denote the number of triples in $X_1 \times X_2 \times X_3$ that are edges of $\h$, and let $d_\h(X_1,X_2,X_3) \coloneqq e_\h(X_1,X_2,X_3)/(\ab{X_1}\ab{X_2}\ab{X_3})$ be the edge density. We say that the triple $(X_1,X_2,X_3)$ is \emph{weakly $\eta$-quasirandom} if, for all $X_i' \subseteq X_i$, we have that
	\[
		\ab{e_\h(X_1',X_2',X_3') - d_\h(X_1,X_2,X_3)\ab{X_1'}\ab{X_2'}\ab{X_3'}} \leq \eta \ab{X_1}\ab{X_2}\ab{X_3}.
	\]
	A partition $\Q_V$ of $V(\h)$ is called \emph{weakly $\eta$-quasirandom} if, for at least a $(1-\eta)$-fraction of the triples $(x_1,x_2,x_3) \in V(\h)^3$, the triple $(\Q_V[x_1],\Q_V[x_2],\Q_V[x_3])$ is weakly $\eta$-quasirandom.
\end{definition}
With this definition, we may state Chung's weak hypergraph regularity lemma \cite{MR1099803}.
\begin{theorem}[{\cite[Theorem 2.1]{MR1099803}}]\label{thm:chung}
	For all $\eta>0$, there exists a constant $K=K(\eta)$ such that the following holds. For every $3$-graph $\h$, there is a weakly $\eta$-quasirandom partition of $V(\h)$ into at most $K$ parts. Moreover, $K \leq \twr(\poly(1/\eta))$.
\end{theorem}
As the names suggest, weak quasirandomness is a weaker notion than quasirandomness (see \cref{prop:strong implies weak} for details). In particular, \cref{thm:chung} follows from \cref{thm:gowers regularity}, except that the latter yields a worse bound on $K$.

\subsection{Formal statement of our main theorem}
We now formally state \cref{thm:main UB informal}, our strengthened regularity lemma for $3$-graphs of bounded VC$_2$ dimension. 
\begin{theorem}\label{thm:main UB formal}
	Fix an integer $d$. For all $\eta>0$ and all increasing polynomial functions $\psi:(0,1) \to (0,1)$ with $\psi(x)\leq x$, the following holds. Let $\h$ be a sufficiently large $3$-graph with VC$_2$ dimension at most $d$. There is a chain partition $\Q=(\Q_V,\Q_E)$ of $\h$ such that, for at least a $(1-\eta)$-fraction of the triples $(x_1,x_2,x_3) \in V(\h)^3$, the chain $(G,\h')\coloneqq \h[\Q;x_1,x_2,x_3]$ satisfies the following properties. The graph $G$ is $\psi(\delta(G))$-quasirandom, and the relative density $d(\h'\mid G)$ satisfies $d(\h'\mid G) \in [0,\eta] \cup [1-\eta,1]$.

	Additionally, 
	\[
		\ab{\Q_V} \leq \twr(\twr(\psi(\eta)^{-C})),
	\]
	where $C$ is a constant depending only on $d$. 
\end{theorem}
As indicated in the introduction, the assumption that $\h$ has bounded VC$_2$ dimension allows us to strengthen \cref{thm:gowers regularity} in two ways. First, and most importantly, the bound on the number of parts is improved from wowzer-type to tower-type. Secondly, as in \cref{thm:FPS}, we obtain a stronger conclusion than most triples of vertices lying in $(\eta,\psi)$-quasirandom chains: most triples now lie in chains whose graph is $\psi(\delta(G))$-quasirandom, but where the relative density of the hypergraph is very close to $0$ or $1$. As in the graph case, it is straightforward to show that this is indeed a strictly stronger property than quasirandomness; for a proof, see e.g.\ \cite[Proposition 2.24]{MR4662634}.

\subsection{Cylinder regularity for hypergraphs}\label{sec:cylinder defs}
The goal of this section is to state our cylinder regularity lemma for $3$-graphs, which is the key new tool we develop in this paper towards the proof of \cref{thm:main UB informal}.
In order to do so, we first need to introduce some further definitions and notation.

\begin{definition}[Cylinders and partitions]\label{def:cylinder}
	Let $X_1,\dots, X_t$ be finite sets, which we think of as the vertex parts of some $t$-partite (hyper)graph. A \emph{vertex subcylinder} is a product set $Y \coloneqq Y_1 \times \dotsb\times Y_t$, where $Y_i \subseteq X_i$ for all $i$. If $Y$ is a vertex subcylinder, we also denote by $Y(X_i)$ the set $Y_i$.

	A \emph{vertex cylinder partition} is a collection $\p_V = \{Y^1,\dots,Y^m\}$, where each $Y^j = Y^j(X_1) \times \dotsb \times Y^j(X_t)$ is a vertex subcylinder, such that the subcylinders $Y^j$ partition the vertex cylinder $X_1 \times \dotsb \times X_t$. In other words, for every $(x_1,\dots,x_t) \in X_1 \times \dotsb \times X_t$, there is a unique $j$ such that $(x_1,\dots,x_t) \in Y^j$. We denote this unique part by $\p_V[x_1,\dots,x_t]$. We denote by $\ab{\p_V}$ the number of parts in $\p_V$.

	An \emph{edge partition}\footnote{We stress that the partition of the edges is not ``cylindrical'': this notion is identical to that of the edge partition appearing in \cref{def:chain partitions}.} $\p_E(Y)$ of a vertex cylinder $Y=Y_1\dtimes Y_t$ consists of, for all $1 \leq i<j \leq t$, a partition $\p_E(Y_i \times Y_j)$ of the complete bipartite graph $Y_i \times Y_j$ into subgraphs $Z_{ij}^1,\dots,Z_{ij}^\ell$. Note that if we select one such bipartite graph for each $1 \leq i<j\leq t$, we obtain a $t$-partite graph; in particular, given $\p_E(Y)$,
	each $t$-tuple $(x_1,\dots,x_t) \in Y_1\dtimes Y_t$ defines a unique such $t$-partite graph, namely the one containing all edges $x_ix_j$. We denote this $t$-partite graph by $\p_E[x_1,\dots,x_t]$. 
	We also define $\ab{\p_E(Y)}\coloneqq \max_{i<j}\ab{\p_E(Y_i\times Y_j)}$ to be the maximum number of bipartite graphs we use in any of our $\binom t2$ edge partitions. 

	Finally, a \emph{cylinder chain partition} is a pair $\p=(\p_V, \p_E)$, where $\p_V$ is a vertex cylinder partition, and $\p_E = \{\p_E(Y): Y \in \p_V\}$ is a set of edge partitions, one for each part of $\p_V$. We denote by $\ab{\p_E}$ the maximum of $\ab{\p_E(Y)}$ over all $Y \in \p_V$. 
\end{definition}

\begin{definition}[Subhypergraphs on chains]
	Let $\h$ be a $t$-partite $3$-graph on vertex parts $X_1,\dots,X_t$. For a vertex subcylinder $Y = Y_1\dtimes Y_t$, we denote by $\h[Y]$ the induced subhypergraph on vertex set $Y_1 \cup \dotsb \cup Y_t$. For a $t$-partite graph $Z$ on $Y_1\cup \dots \cup Y_t$, we further denote by $\h[Y;Z]$ the subhypergraph of $\h[Y]$ comprising all edges supported on $\Delta(Z)$.

	Note that $\h[Y;Z]$ can itself be viewed as a union of $\binom t3$ chains. Namely, for every $\{i,j,k\} \in \binom{[t]}3$, we have the chain with vertex sets $Y_i, Y_j, Y_k$, edge sets $Z_{ij}, Z_{ik}, Z_{jk}$, and the hyperedges of $\h$ which are supported on $\Delta(Z_{ij} \cup Z_{ik} \cup Z_{jk})$. We denote this chain by $\h[Y;Z;\{i,j,k\}]$.
\end{definition}

\noindent We now extend the quasirandomness notions discussed above to cylinders and cylinder partitions.
\begin{definition}[Quasirandomness of cylinders]
	Let $Y=Y_1\times \dotsb \times Y_t$ be a vertex cylinder.
	We say that an edge partition $\p_E(Y)$ is $\alpha$-quasirandom if, whenever $(x_1,\dots,x_t) \in Y_1\dtimes Y_t$ is picked uniformly at random, then the $t$-partite graph $\p_E[x_1,\dots,x_t]$ is $\alpha$-quasirandom with probability at least $1-\alpha/2$.

	Let $\h$ be a $t$-partite $3$-graph on vertex parts $X_1,\dots,X_t$.
	Given a cylinder chain partition $\p=(\p_V,\p_E)$, and given $\eta>0$ and $\psi:(0,1) \to (0,1)$, we say that $\h$ is $(\eta,\psi)$-quasirandom relative to $\p$ if the following holds. When $(x_1,\dots,x_t) \in X_1 \dtimes X_t$ is picked uniformly at random, with probability at least $1-\eta$ we have that $\h[\p_V[x_1,\dots,x_t]; \p_E[x_1,\dots,x_t]]$ is $(\eta,\psi)$-quasirandom, in the sense of \cref{def:product density}.
\end{definition}

We are now ready to state our cylinder regularity lemma for $3$-graphs, which will be the main technical tool used in the proof of \cref{thm:main UB formal}.
In the statement, we will use the following simple extension of the tower function used above, where we denote by $\twr(\tau,x)$ a tower of twos of height $\tau$ with an $x$ as the top exponent. Formally, we recursively define this function by $\twr(0,x)=x$ and $\twr(\tau+1,x)=2^{\twr(\tau,x)}$. Note that $\twr(\tau,1)=\twr(\tau)$, that is, this agrees with the previous definition upon setting $x=1$.
\begin{theorem}\label{thm:hyper cylinder}
	Fix an integer $t$, a parameter $\eta>0$, and an increasing polynomial function $\psi:(0,1) \to (0,1)$ satisfying $\psi(x) \leq 2^{-100}x^{28}$. Let $\h$ be a $t$-partite $3$-graph. Then there exists a cylinder chain partition $\p = (\p_V, \p_E)$, such that $\h$ is $(\eta, \psi)$-quasirandom relative to $\p$, and such that
	\[
		\ab{\p_V}  \leq \twr\left(\frac{10^5 t^6}{\eta^3},\psi\left(\frac\eta{t^2}\right)^{-20}\right)
	\]
	and
	\[
		\ab{\p_E} \leq \twr\left(\frac{10^5 t^6}{\eta^3},\left(\frac{t^2}\eta\right)^{10}\right).
	\]
\end{theorem}

\section{Cylinder regularity}\label{sec:cylinder}
Our goal in this section is to prove \cref{thm:hyper cylinder}. In the first subsection below we recall some basic facts and tools. We then use these tools to prove \cref{thm:hyper cylinder} in \cref{subsec:cylinder proof}. We stress that \cref{thm:hyper cylinder} holds for all 3-graphs, and in particular, in this section we do not assume that the 3-graphs we are dealing with have bounded VC$_2$ dimension.
\subsection{Mean-squared density}
In this subsection, we record various facts about the mean-squared density of a cylinder chain partition, which we now define.

First, we extend our notion of edge partition (from \cref{def:cylinder}) as follows. Let $G$ be a tripartite graph with parts $Y_1,Y_2,Y_3$, comprising three bipartite graphs $G_{12}\subseteq Y_1 \times Y_2, G_{13} \subseteq Y_1 \times Y_3, G_{23} \subseteq Y_2 \times Y_3$. An \emph{edge partition} $\p_E(G)$ consists of a partition of each of the three bipartite graphs $G_{12},G_{13},G_{23}$ into subgraphs. We denote by $\p_E(G;Y_1 \times Y_2)$ the partition of $G_{12}$, and similarly for the other two pairs of parts. As before, $\ab{\p_E(G)}\coloneqq\max\{\ab{\p_E(G;Y_1 \times Y_2)},\ab{\p_E(G;Y_1 \times Y_3)},\ab{\p_E(G;Y_2 \times Y_3)}\}$ denotes the maximum number of parts partitioning the bipartite graphs. Note that all of these notions extend our previous notion
of an edge partition $\p_E(Y)$ of a tripartite cylinder; the previous case corresponds to the case when $G$ is the complete tripartite graph on $Y_1 \cup Y_2 \cup Y_3$. As we will be dealing with such complete tripartite graphs repeatedly in what follows, we also establish the following notation: $K(Y_1,Y_2,Y_3)$ denotes the complete tripartite graph on $Y_1 \cup Y_2 \cup Y_3$.

We recall that if $(G,\h)$ is a tripartite chain, we denote by
\[
	d(\h \mid G) = \frac{\ab{E(\h)}}{\ab{\Delta(G)}}
\]
the relative density of $\h$ on $G$, with the convention that $0/0=0$. In case $G'$ is a tripartite subgraph of $G$, we similarly define
\[
	d(\h \mid G') \coloneqq \frac{\ab{E(\h)\cap \Delta(G')}}{\ab{\Delta(G')}}
\]
to be the relative density of the chain $(G',\h')$, where $\h'$ is the subhypergraph of $\h$ supported on $\Delta(G')$.

\begin{definition}[Mean-squared density\footnote{We remark that there are several different reasonable definitions of the mean-squared density; for example, rather than weighting the summands in \eqref{eq:mean-squared} by $\ab{\Delta(Z_{12} \cup Z_{13} \cup Z_{23})}$, we could weight by its approximation $\delta(Z_{12} \cup Z_{13}\cup Z_{23}) \ab{Z_{12}} \ab{Z_{13}} \ab{Z_{23}}$ (recall \cref{def:product density}). The definition we use is a generalization to cylinder partitions of the one used by Gowers \cite{MR2195580}, and is more technically convenient for our purposes.}]\label{def:mean-squared}
	Let $(G,\h)$ be a tripartite chain with vertex parts $Y_i \cup Y_j \cup Y_k$, and let $\p_E(G)$ be an edge partition of $G$. The \emph{mean-squared density} of $\p_E(G)$ is defined to be
	\begin{equation}\label{eq:mean-squared}
		q(\p_E(G)) \coloneqq \sum_{\substack{Z_{ij} \in \p_E(G;Y_i \times Y_j)\\Z_{ik} \in \p_E(G;Y_i \times Y_k)\\Z_{jk} \in \p_E(G;Y_j \times Y_k)}} \frac{\ab{\Delta(Z_{ij}\cup Z_{ik}\cup Z_{jk})}}{\ab{\Delta(G)}} d(\h \mid Z_{ij}\cup Z_{ik}\cup Z_{jk})^2.
	\end{equation}
	In words, we sum over all tripartite subgraphs $G'$ of $G$ obtained by picking one element of each of the partitions of the three bipartite subgraphs. For each such $G'$, we square the relative density of $\h$, and then weight this summand by the fraction of triangles of $G$ that appear in $G'$. Note that if $\p_E(G)$ is the trivial edge partition, then $q(\p_E(G))=d(\h \mid G)^2$.

	Let $\h$ now be a $t$-partite $3$-graph with vertex set $X_1\cup \dotsb\cup X_t$. Suppose $Y = Y_1 \dtimes Y_t$ is a vertex cylinder and $\p_E(Y)$ is an edge partition as in \cref{def:cylinder}. The mean-squared density of the cylinder $Y$ is then defined to be
	\begin{equation}\label{eq:q vertex cylinder}
		q(\p_E(Y)) \coloneqq \sum_{1 \leq i<j<k\leq t} q(\p_E(K(Y_i,Y_j,Y_k))),
	\end{equation}
	where we recall that $K(Y_i,Y_j,Y_k)$ is the complete tripartite graph on $Y_i \cup Y_j \cup Y_k$.
	Finally, given a cylinder chain partition $\p=(\p_V,\p_E)$ of $X_1\dtimes X_t$, we define its mean-squared density to be
	\[
		q(\p) \coloneqq \sum_{Y \in \p_V} \left( \prod_{i=1}^t \frac{\ab{Y_i}}{\ab{X_i}} \right) q(\p_E(Y)).
	\]
\end{definition}

The basic observations we need about the mean-squared density are standard, namely that it is bounded and monotonic with respect to refinement. To formally state these properties, let us say that a vertex cylinder partition $\p_V'$ \emph{refines} another vertex cylinder partition $\p_V$ if, for every $Y' \in \p_V'$, there is a unique $Y \in \p_V$ containing $Y'$. That is, if $Y' = Y_1'\dtimes Y_t'$ and $Y = Y_1 \dtimes Y_t$, then $Y_i' \subseteq Y_i$ for all $i$. Similarly, an edge partition $\p_E'(G)$ \emph{refines} another edge partition $\p_E(G)$ of the same $t$-partite graph if for all pairs of parts $Y_i,Y_j$, each part of $\p_E'(G;Y_i \times Y_j)$ lies in a unique part of $\p_E(G;Y_i \times Y_j)$. For two edge partitions $\p_E(Y),\p_E'(Y)$ of a vertex cylinder, we say that $\p_E'(Y)$ refines $\p_E(Y)$ if it refines it in the sense above, where we view it as an edge partition of the complete $t$-partite graph. Finally, let us say that one cylinder chain partition $\p'=(\p_V',\p_E')$ \emph{refines} another cylinder chain partition $\p=(\p_V,\p_E)$ if $\p_V'$ refines $\p_V$ in the sense above, and if $\p_E'$ refines $\p_E$ in the following sense. If $Y' \in \p_V'$, let $Y$ be the part of $\p_V$ containing $Y'$. For all $i<j$, the partition $\p_E(Y_i \times Y_j)$ naturally restricts to a partition $\p_E(Y_i \times Y_j)|_{Y_i' \times Y_j'}$ of the complete bipartite graph $Y_i' \times Y_j'$. We then say that
$\p_E'$ refines $\p_E$ if, for all $Y' \in \p_V'$ and for all $i<j$, the partition $\p_E'(Y_i' \times Y_j')$ refines $\p_E(Y_i \times Y_j)|_{Y_i' \times Y_j'}$.

We may now state the properties we will use of the mean-squared density.
\begin{lemma}\label{lem:q properties}
	Let $(G,\h)$ be a tripartite chain. 
	\begin{lemenum}
	
	\item 
	If $\p_E(G)$ is an edge partition of $G$, then $0\leq q(\p_E(G))\leq 1$. \label{lemit:G01}
	
	\item  If $\p_E'(G)$ is a refinement of $\p_E(G)$, then $q(\p_E'(G))\geq q(\p_E(G))$.\label{lemit:G refinement}
	\end{lemenum}

	Similarly, let $\h$ be a $t$-partite $3$-graph. 
	\begin{lemenum}[start=3]
		
		\item 
		If $\p=(\p_V,\p_E)$ is a cylinder chain partition, then $0 \leq q(\p) \leq \binom t3$. \label{lemit:P0t3}
		
		\item If $\p'$ is a refinement of $\p$, then $q(\p') \geq q(\p)$.\label{lemit:P refinement}
	\end{lemenum}
\end{lemma}
\cref{lem:q properties} is proved by several elementary but tedious applications of the Cauchy--Schwarz inequality, and we thus defer the proof to \cref{appendix:mean-squared}.

As in the majority of proofs of regularity lemmas, our proof of \cref{thm:hyper cylinder} will proceed by an energy increment approach. Concretely, given a cylinder chain partition of $\h$, we will either find that it is sufficiently quasirandom, or else find a refinement of it that increases the mean-squared density by a constant amount (independent of $\h$ or of the given partition). Each single step of this iteration will be obtained by the following lemma of Gowers \cite{MR2195580}.
It states that if a tripartite chain is not $\eta$-quasirandom, then we can find an edge partition of it for which the mean-squared density increases appreciably, while the number of parts in the new edge partition is not too large.

\begin{lemma}[{\cite[Lemma 8.4]{MR2195580}}]\label{lem:one cylinder}
	Consider a tripartite graph $G$ with vertex sets $Y_1,Y_2,Y_3$ and edge sets $Z_{12},Z_{13},Z_{23}$. Let $\delta=\delta(G)$. Suppose that each of the three bipartite graphs is $\alpha$-quasirandom, where $\alpha \leq 2^{-100}\delta^{28}$.

	Let $(G,\h)$ be a tripartite chain, and suppose that $\h$ is not $\eta$-quasirandom relative to $G$. Let $d=d(\h \mid G)$ be the relative density of $\h$. Then there is an edge partition $\p_E(G)$ such that $q(\p_E(G)) \geq d^2 + 2^{-10}\eta^2$. Moreover, $\ab{\p_E(G)} \leq 3^{\delta^{-4}}$.
\end{lemma}

Using \cref{lem:one cylinder}, we are able to take an irregular cylinder chain partition and further partition each of the irregular edge parts in order to increase the mean-squared density. However, we cannot then directly iterate \cref{lem:one cylinder}, as the assumption of \cref{lem:one cylinder} requires that each bipartite graph is $\alpha$-quasirandom, and this condition may be destroyed by the previous partitioning step. To get around this, we will apply the cylinder regularity lemma for graphs \cite{MR1333857} at each step. We will actually use the following, slightly generalized, version of the cylinder regularity lemma for graphs, which allows us to simultaneously regularize a family of graphs on the same vertex set.
\begin{lemma}\label{lem:graph cylinder}
	Let $G_1,\dots,G_m$ be $t$-partite graphs on the same vertex set. For any $\alpha>0$, there exists a vertex cylinder partition $\p_V$ such that the following holds. If $(x_1,\dots,x_t)$ is chosen uniformly at random from $X_1\dtimes X_t$, then with probability at least $1-\alpha/2$, we have that for all $i \in [m]$, the subgraph of $G_i$ induced on the vertex set $\p_V[x_1,\dots,x_t]$ is $\alpha$-quasirandom. Moreover, $\ab{\p_V} \leq 2^{mt^2\alpha^{-20}}$.
\end{lemma}
The only differences between \cref{lem:graph cylinder} and the cylinder regularity lemma proved by Duke, Lefmann, and R\"odl \cite{MR1333857} are that we are regularizing multiple graphs, and that we use a different notion of quasirandomness, leading to the different power of $\alpha$ (since these two notions of graph quasirandomness are polynomially related). It is well known, and used frequently in hypergraph regularity theory, that one can simultaneously regularize multiple graphs on the same vertex set; see e.g.\ \cite[Theorem 7.8]{MR2195580} or \cite[Lemma 3.7]{MR1884430}. For a proof of the cylinder regularity lemma for multiple graphs, see \cite[Theorem 7.2]{MR4170446}.

\subsection{Proof of Theorem \ref{thm:hyper cylinder}}\label{subsec:cylinder proof}

With these preliminaries, we are ready to prove \cref{thm:hyper cylinder}. In fact, for future convenience (see \cref{lem:all cylinders 01}), we prove the following slight strengthening of \cref{thm:hyper cylinder}.
\begin{theorem}\label{thm:hyper cylinder with delta}
	Fix an integer $t$, a parameter $\eta>0$, and an increasing polynomial function $\psi:(0,1) \to (0,1)$ satisfying $\psi(x) \leq 2^{-100}x^{28}$. Let $\h$ be a $t$-partite $3$-graph. Then there exists a cylinder chain partition $\p = (\p_V, \p_E)$, such that $\h$ is $(\eta, \psi)$-quasirandom relative to $\p$, and such that
	\[
		\ab{\p_V}  \leq \twr\left(\frac{10^5 t^6}{\eta^3},\psi\left(\frac\eta{t^2}\right)^{-20}\right)
	\]
	and
	\[
		\ab{\p_E} \leq \twr\left(\frac{10^5 t^6}{\eta^3},\left(\frac{t^2}\eta\right)^{10}\right).
	\]
	Furthermore, there is a constant $\delta>0$, depending only on $t,\eta$, and $\psi$, such that a $(1-\eta)$-fraction of tuples $(x_1,\dots,x_t)$ lie in a chain with product density at least $\delta$.
\end{theorem}
The only difference between \cref{thm:hyper cylinder,thm:hyper cylinder with delta} is the final sentence, stating that most $t$-tuples lie in chains whose product density is not too small. This will be useful for us later, and arises naturally from the proof.

\begin{proof}[Proof of \cref{thm:hyper cylinder with delta}]
	We begin by defining functions $A,B:\N \to \N$ as follows. We set $A(0) = B(0)= 1$, and recursively define
	\begin{align*}
		B(\tau+1) &= 2^{2^{(t^2 B(\tau)/\eta)^{2t^2}}},\\
		A(\tau+1) &=2^{B(\tau+1)t^4\psi(\eta/(t^2 B(\tau+1)))^{-20t^2}}A(\tau) .
	\end{align*}
	The fact that $\psi$ is a polynomial function implies that
	\begin{equation}\label{eq:AB UB}
		A(\tau) \leq \twr\left(3\tau,\psi\left(\frac\eta{t^2}\right)^{-20}\right) \qquad \text{ and } \qquad
		B(\tau) \leq \twr\left(3\tau,\left(\frac{t^2}\eta\right)^{10}\right).
	\end{equation}
	We remark that similar bounds would hold if $\psi$ were any function such that $1/\psi$ is of constant tower height; the only difference is that the tower height of $A(\tau)$ would be at most $C\tau$, for some constant $C$ depending on $\psi$. In case $\psi$ is a polynomial, we may take $C=3$, as stated in \eqref{eq:AB UB}.

	We will define a sequence $\p\up 0, \p \up 1, \dots$ of cylinder chain partitions, and inductively maintain the following properties, where $X_1,\dots,X_t$ are the $t$ vertex parts of $\h$. 
	\begin{enumerate}[label=(\roman*)]
		\item $\ab{\p_V\up \tau} \leq A(\tau)$. \label{it:PV UB}
		\item $\ab{\p_E\up \tau} \leq B(\tau)$. \label{it:PE UB}
		\item $q(\p\up \tau) \geq 2^{-10} \eta^3 \tau/t^3$. \label{it:q LB}
		\item Let $\delta=(\eta/(t^2B(\tau)))^{\binom t2}$ and $\alpha=\psi(\delta)$. If we sample $(x_1,\dots,x_t) \in X_1 \dtimes X_t$ uniformly at random,
		the following happens with probability at least $1-\eta/2$: the $t$-partite graph $Z\coloneqq \p_E\up \tau[x_1,\dots,x_t]$ satisfies $\delta(Z) \geq \delta$ and $Z$ is $\alpha$-quasirandom.
		\label{it:edge quasirandom}
		\item Either $\h$ is $(\eta,\psi)$-quasirandom with respect to $\p\up \tau$, or else we can define $\p\up{\tau+1}$ satisfying the above properties. \label{it:stop}
	\end{enumerate}

	To begin the induction, we let $\p_V\up 0$ be the trivial cylinder partition into one part $Y^0=X_1 \dtimes X_t$, and let the corresponding edge partition $\p_E(Y^0(X_i) \times Y^0(X_j))$ also be the trivial partition into one part $X_i \times X_j$, for all $1 \leq i<j\leq t$. Then it is clear that the claimed properties \ref{it:PV UB}--\ref{it:edge quasirandom} hold. 

	Inductively, suppose we have defined $\p\up\tau$ satisfying properties \ref{it:PV UB}--\ref{it:edge quasirandom}. We are done if $\h$ is $(\eta,\psi)$-quasirandom with respect to $\p\up \tau$, so we may assume that this is not the case.

	As in property \ref{it:edge quasirandom}, let $\delta=(\eta/(t^2B(\tau)))^{\binom t2}$ and let $\alpha = \psi(\delta)$. The fact that $\h$ is not $(\eta,\psi)$-quasirandom with respect to $\p\up \tau$ means that if we pick $(x_1,\dots,x_t) \in X_1 \dtimes X_t$ uniformly at random, then with probability at least $\eta$, the chain $\h[\p_V\up \tau[x_1,\dots,x_t];\p_E\up \tau[x_1,\dots,x_t]]$ is not $(\eta,\psi)$-quasirandom. That, in turn, means that with probability at least $\eta$, either the chain $\h[\p_V\up \tau[x_1,\dots,x_t];\p_E\up \tau[x_1,\dots,x_t]]$ is not $\eta$-quasirandom, or that the $t$-partite graph $\p_E\up \tau[x_1,\dots,x_t]$ is not $\psi(\delta(\p_E\up \tau[x_1,\dots,x_t]))$-quasirandom.

	In this random experiment, we know by property \ref{it:edge quasirandom} that with probability at least $1-\eta/2$, the $t$-partite graph $\p_E\up \tau[x_1,\dots,x_t]$ satisfies $\delta(\p_E\up \tau[x_1,\dots,x_t]) \geq \delta$ and is $\alpha$-quasirandom. Since $\psi$ is increasing, this in turn implies that $\p_E\up \tau[x_1,\dots,x_t]$ is $\psi(\delta(\p_E\up \tau[x_1,\dots,x_t]))$-quasirandom. Therefore, with probability at least $\eta-\eta/2 = \eta/2$, we have that $\p_E\up \tau[x_1,\dots,x_t]$ is $\psi(\delta(\p_E\up \tau[x_1,\dots,x_t]))$-quasirandom, whereas the chain $\h[\p_V\up \tau[x_1,\dots,x_t];\p_E\up \tau[x_1,\dots,x_t]]$ is not $\eta$-quasirandom. By the pigeonhole principle, there exist distinct indices $a,b,c \in [t]$ such that for at least an $\eta/(2\binom t3)$-fraction of the tuples $(x_1,\dots,x_t) \in X_1 \dtimes X_t$, the tripartite chain 
	\begin{equation}
		\h[\p_V\up \tau[x_1,\dots,x_t];\p_E\up \tau[x_1,\dots,x_t]; \{a,b,c\}]\label{eq:useful}
	\end{equation}
	is not $\eta$-quasirandom but does satisfy that $\delta(\p_E\up \tau[x_1,\dots,x_t]) \geq \delta$ and that $\p_E\up \tau[x_1,\dots,x_t]$ is $\psi(\delta(\p_E\up \tau[x_1,\dots,x_t]))$-quasirandom. We call such $t$-tuples $(x_1,\dots,x_t)$ \emph{useful}. Extending this terminology, we also call the tripartite chain in \eqref{eq:useful} \emph{useful}.

	Fix a useful $t$-tuple $(x_1,\dots,x_t)$, and let $Y = \p_V\up \tau[x_1,\dots,x_t]$ and $Z = \p_E\up \tau[x_1,\dots,x_t]$. Let $G$ be the tripartite graph with parts $Y_a \cup Y_b \cup Y_c$ and edge sets $Z_{ab} \cup Z_{ac} \cup Z_{bc}$. Note that $\delta(G) \geq \delta(Z) \geq \delta$ and that $G$ is $\psi(\delta(G))$-quasirandom, as $\psi$ is increasing and by the definition of a useful tuple.
	We may thus apply \cref{lem:one cylinder} to $G$, using our assumptions that $\psi$ is increasing and satisfies $\psi(x) \leq 2^{-100}x^{28}$. Doing this gives us, for each such useful chain, an edge partition of $G$ into at most $3^{\delta(G)^{-4}}\leq 3^{\delta(Z)^{-4}} \leq 3^{\delta^{-4}}$ parts.
	
	We now do this process for all useful chains, and use this step to define a new edge partition $\p_E\up {\tau+\frac 12}$ (for the moment, we will not change $\p_V\up \tau$). To do so,
	let us fix a vertex cylinder $Y \in \p_V\up \tau$, and define the new edge partition $\p_E\up{\tau+\frac 12}(Y)$ as follows. 
	Let the bipartite graphs in $\p_E(Y_a \times Y_b)$ be $Z_{ab}^1,\dots,Z_{ab}^\ell$. For each $1 \leq \mu \leq \ell$, we consider the set of all tripartite graphs arising by adding to $Z_{ab}^\mu$ some part of $\p_E\up \tau(Y_a \times Y_c)$ and $\p_E\up \tau (Y_b \times Y_c)$. Some of these tripartite chains are useful and hence subpartitioned by \cref{lem:one cylinder}, and we then let $Z_{ab}^\mu$ be partitioned as the common refinement of all of these partitions. By doing this for all $1 \leq \mu \leq \ell$, we obtain a new edge partition $\p_E\up{\tau+\frac 12}(Y_a \times Y_b)$. We do the same to obtain the edge partitions $\p_E\up{\tau+\frac 12}(Y_a \times Y_c),\p_E\up{\tau+\frac 12}(Y_b \times Y_c)$. For all $ij \notin \{ab,ac,bc\}$, we set $\p_E\up{\tau+\frac 12}(Y_i \times Y_j) = \p_E\up \tau (Y_i \times Y_j)$. Finally, doing this for all $Y \in \p_V\up \tau$ defines our new edge partition $\p_E\up{\tau+\frac 12}$. 
	
	We now claim that
	\begin{equation}\label{eq:density boost}
		q((\p_V\up \tau, \p_E\up{\tau+\frac 12})) \geq q((\p_V\up \tau, \p_E\up \tau)) + 2^{-10} \eta^3/t^3.
	\end{equation}
	Indeed, we first note that as we did not change the vertex partition,
	\begin{equation}\label{eq:pass to Y}
		q((\p_V\up \tau, \p_E\up{\tau+\frac 12})) - q((\p_V\up \tau, \p_E\up \tau)) = \sum_{Y \in \p_V\up \tau} \left( \prod_{1 \leq i \leq t} \frac{\ab{Y_i}}{\ab{X_i}} \right) \left(q(\p_E\up{\tau + \frac 12}(Y))-q(\p_E\up \tau(Y))\right).
	\end{equation}
	To estimate the inner difference, let us fix some $Y \in \p_V\up \tau$. We then have by definition that
	\[
		q(\p_E\up{\tau + \frac 12}(Y))-q(\p_E\up \tau(Y)) =
		\sum_{1 \leq i<j<k\leq t}\left[q(\p_E\up{\tau + \frac 12}(K(Y_i,Y_j,Y_k)))-q(\p_E\up \tau(K(Y_i,Y_j,Y_k)))\right].
	\]
	For each triple $(i,j,k) \neq (a,b,c)$, we recall that the edge partition $\p_E\up{\tau + \frac 12}(K(Y_i,Y_j,Y_k))$ refines $\p_E\up{\tau}(K(Y_i,Y_j,Y_k))$, hence $q(\p_E\up{\tau + \frac 12}(K(Y_i,Y_j,Y_k)))-q(\p_E\up \tau(K(Y_i,Y_j,Y_k)))\geq 0$ by \cref{lemit:G refinement}. Therefore,
	\begin{equation}\label{eq:ijk abc}
		q(\p_E\up{\tau + \frac 12}(Y))-q(\p_E\up \tau(Y)) \geq q(\p_E\up{\tau + \frac 12}(K(Y_a,Y_b,Y_c)))-q(\p_E\up \tau(K(Y_a,Y_b,Y_c))).
	\end{equation}
	By definition, we have that
	\begin{equation}\label{eq:q tau+1/2}
		q(\p_E\up{\tau + \frac 12}(K(Y_a,Y_b,Y_c)))=\sum_{\substack{Z'_{ab}\in \p_E\up{\tau+\frac 12}(Y_a \times Y_b)\\Z'_{ac}\in \p_E\up{\tau+\frac 12}(Y_a \times Y_c)\\Z'_{bc}\in \p_E\up{\tau+\frac 12}(Y_b \times Y_c)}} \frac{\ab{\Delta(Z'_{ab}\cup Z'_{ac}\cup Z'_{bc})}}{\ab{Y_a}\ab{Y_b}\ab{Y_c}}d(\h \mid Z'_{ab}\cup Z'_{ac}\cup Z'_{bc})^2.
	\end{equation}
	Recall that $\p_E\up{\tau+\frac 12}(Y_a \times Y_b)$ refines $\p_E\up{\tau}(Y_a \times Y_b)$, and similarly for $Y_a \times Y_c$ and $Y_b \times Y_c$. This means that we may rearrange the sum in \eqref{eq:q tau+1/2} to first sum over parts in $\p_E \up \tau$ and thus find that $q(\p_E\up{\tau + \frac 12}(K(Y_a,Y_b,Y_c)))$ equals
	\begin{equation*}
		\sum_{\substack{Z_{ab}\in \p_E\up{\tau}(Y_a \times Y_b)\\Z_{ac}\in \p_E\up{\tau}(Y_a \times Y_c)\\Z_{bc}\in \p_E\up{\tau}(Y_b \times Y_c)}}\sum_{\substack{Z'_{ab}\in \p_E\up{\tau+\frac 12}(Y_a \times Y_b)\\Z'_{ac}\in \p_E\up{\tau+\frac 12}(Y_a \times Y_c)\\Z'_{bc}\in \p_E\up{\tau+\frac 12}(Y_b \times Y_c)\\Z_{ab}' \subseteq Z_{ab},Z_{ac}'\subseteq Z_{ac},Z_{bc}'\subseteq Z_{bc}}} \frac{\ab{\Delta(Z'_{ab}\cup Z'_{ac}\cup Z'_{bc})}}{\ab{Y_a}\ab{Y_b}\ab{Y_c}}d(\h \mid Z'_{ab}\cup Z'_{ac}\cup Z'_{bc})^2.
	\end{equation*}
	This, in turn, is equal to
	\[
		\sum_{\substack{Z_{ab}\in \p_E\up{\tau}(Y_a \times Y_b)\\Z_{ac}\in \p_E\up{\tau}(Y_a \times Y_c)\\Z_{bc}\in \p_E\up{\tau}(Y_b \times Y_c)}}
		\frac{\ab{\Delta(Z_{ab}\cup Z_{ac}\cup Z_{bc})}}{\ab{Y_a}\ab{Y_b}\ab{Y_c}}\hspace{-7.6pt}
		\sum_{\substack{Z'_{ab}\in \p_E\up{\tau+\frac 12}(Y_a \times Y_b)\\Z'_{ac}\in \p_E\up{\tau+\frac 12}(Y_a \times Y_c)\\Z'_{bc}\in \p_E\up{\tau+\frac 12}(Y_b \times Y_c)\\Z_{ab}' \subseteq Z_{ab},Z_{ac}'\subseteq Z_{ac},Z_{bc}'\subseteq Z_{bc}}} \frac{\ab{\Delta(Z'_{ab}\cup Z'_{ac}\cup Z'_{bc})}}{\ab{\Delta(Z_{ab}\cup Z_{ac}\cup Z_{bc})}}d(\h \mid Z'_{ab}\cup Z'_{ac}\cup Z'_{bc})^2.
	\]
	Crucially, we now note that this inner sum is exactly the type of quantity appearing in \eqref{eq:mean-squared}. More precisely, suppose we fix 
	some $Z_{ab} \in \p_E\up \tau(Y_a \times Y_b), Z_{ac} \in \p_E\up \tau(Y_a \times Y_c),Z_{bc} \in \p_E\up \tau(Y_b\times Y_c)$, and define an edge partition $\wt{\p_E}(Z_{ab}\cup Z_{ac}\cup Z_{bc})$ of the tripartite graph they form by simply restricting $\p_E\up{\tau+\frac 12}$ to this tripartite graph. Then the inner sum above is exactly $q(\wt{\p_E}(Z_{ab}\cup Z_{ac}\cup Z_{bc}))$, as defined in \eqref{eq:mean-squared}.
	Therefore, we can rewrite the above as
	\[
		q(\p_E\up{\tau+\frac 12}(K(Y_a,Y_b,Y_c))) = \sum_{\substack{Z_{ab}\in \p_E\up{\tau}(Y_a \times Y_b)\\Z_{ac}\in \p_E\up{\tau}(Y_a \times Y_c)\\Z_{bc}\in \p_E\up{\tau}(Y_b \times Y_c)}}
		\frac{\ab{\Delta(Z_{ab}\cup Z_{ac}\cup Z_{bc})}}{\ab{Y_a}\ab{Y_b}\ab{Y_c}} q(\wt{\p_E}(Z_{ab}\cup Z_{ac}\cup Z_{bc})).
	\]
	For such $Z_{ab},Z_{ac},Z_{bc}$, we also define another edge partition $\p_E'(Z_{ab}\cup Z_{ac}\cup Z_{bc})$ of the same tripartite graph as follows. If the chain $(Z_{ab}\cup Z_{ac}\cup Z_{bc},\h)$ is not useful, we let $\p_E'(Z_{ab}\cup Z_{ac}\cup Z_{bc})$ be the trivial edge partition where each bipartite graph is partitioned into a single part. On the other hand, if this chain is useful, then $\p_E'(Z_{ab}\cup Z_{ac}\cup Z_{bc})$ is defined to be the edge partition given to us by \cref{lem:one cylinder}. Note that in either case, $\wt{\p_E}(Z_{ab}\cup Z_{ac}\cup Z_{bc})$ refines the edge partition $\p_E'(Z_{ab}\cup Z_{ac}\cup Z_{bc})$; indeed, $\p_E\up{\tau+\frac 12}$ was defined to be the common refinement of all partitions arising in this way. Hence, applying \cref{lemit:G refinement} to each tripartite graph $Z_{ab}\cup Z_{ac}\cup Z_{bc}$, we find that
	\[
		q(\p_E\up{\tau+\frac 12}(K(Y_a,Y_b,Y_c))) \geq \sum_{\substack{Z_{ab}\in \p_E\up{\tau}(Y_a \times Y_b)\\Z_{ac}\in \p_E\up{\tau}(Y_a \times Y_c)\\Z_{bc}\in \p_E\up{\tau}(Y_b \times Y_c)}}
		\frac{\ab{\Delta(Z_{ab}\cup Z_{ac}\cup Z_{bc})}}{\ab{Y_a}\ab{Y_b}\ab{Y_c}} q(\p_E'(Z_{ab}\cup Z_{ac}\cup Z_{bc})).
	\]
	Plugging in the definition of $q(\p_E\up \tau(K(Y_a,Y_b,Y_c)))$, we conclude that
	\begin{multline*}
		q(\p_E\up{\tau + \frac 12}(K(Y_a,Y_b,Y_c)))-q(\p_E\up \tau(K(Y_a,Y_b,Y_c))) \\\geq \sum_{\substack{Z_{ab}\in \p_E\up{\tau}(Y_a \times Y_b)\\Z_{ac}\in \p_E\up{\tau}(Y_a \times Y_c)\\Z_{bc}\in \p_E\up{\tau}(Y_b \times Y_c)}}\frac{\ab{\Delta(Z_{ab}\cup Z_{ac}\cup Z_{bc})}}{\ab{Y_a}\ab{Y_b}\ab{Y_c}} \left[ q(\p_E'(Z_{ab}\cup Z_{ac}\cup Z_{bc}))- d(\h \mid Z_{ab}\cup Z_{ac}\cup Z_{bc})^2 \right].
	\end{multline*}
	In the sum above, if the choice of $Z_{ab},Z_{ac},Z_{bc}$ is such that $(Z_{ab}\cup Z_{ac}\cup Z_{bc}, \h)$ is not useful, the difference inside the sum equals zero, as $\p_E'(Z_{ab}\cup Z_{ac}\cup Z_{bc})$ is then a trivial partition, whose mean-squared density equals $d(\h \mid Z_{ab}\cup Z_{ac}\cup Z_{bc})^2$ (recall \cref{def:mean-squared}). For the remaining summands, $\p_E'(Z_{ab}\cup Z_{ac}\cup Z_{bc})$ is defined to be the output of \cref{lem:one cylinder}, and hence the difference inside the sum is at least $2^{-10}\eta^2$. Combining this with \eqref{eq:pass to Y} and \eqref{eq:ijk abc}, we conclude that
	\begin{multline*}
		q((\p_V\up \tau, \p_E\up{\tau+\frac 12})) - q((\p_V\up \tau, \p_E\up \tau)) \\
		\geq 2^{-10}\eta^2 \cdot\sum_{Y \in \p_V\up \tau} \left( \prod_{1 \leq i \leq t} \frac{\ab{Y_i}}{\ab{X_i}} \right)\sum_{\substack{Z_{ab}\in \p_E\up{\tau}(Y_a \times Y_b)\\Z_{ac}\in \p_E\up{\tau}(Y_a \times Y_c)\\Z_{bc}\in \p_E\up{\tau}(Y_b \times Y_c)}}\frac{\ab{\Delta(Z_{ab}\cup Z_{ac}\cup Z_{bc})}}{\ab{Y_a}\ab{Y_b}\ab{Y_c}}\cdot  \bm 1_{(Z_{ab}\cup Z_{ac}\cup Z_{bc},\h)\text{ is useful}}.
	\end{multline*}
	The final observation is that this sum (i.e.\ the right-hand expression without the factor $2^{-10}\eta^2$) is precisely the probability that a uniformly random tuple $(x_1,\dots,x_t) \in X_1 \dtimes X_t$ is useful. By assumption, this probability is at least $\eta/(2\binom t3)\geq \eta/t^3$, and this concludes the proof of the claim \eqref{eq:density boost}.

	In order to estimate $\ab{\p_E\up{\tau+\frac 12}}$, let $\ell = \ab{\p_E\up{\tau}(Y)}$ for some vertex cylinder $Y$. Let $\p_E\up \tau(Y_a \times Y_b) = \{Z_{ab}^1,\dots,Z_{ab}^\ell\}$, and consider some $Z_{ab}^\mu$. There are at most $\ell^2$ choices of bipartite graphs in $\p_E\up \tau(Y_a \times Y_c), \p_E\up \tau (Y_b \times Y_c)$ for which the corresponding chain is useful. For each of these $\ell^2$ choices, \cref{lem:one cylinder} partitions $Z_{ab}^\mu$ into at most $3^{\delta^{-4}}$ subgraphs, and then $\p_E\up{\tau+\frac 12}(Y_a \times Y_b)$ is obtained by taking the common refinement of all of these. This implies that $\ab{\p_E\up{\tau+\frac 12}(Y_a \times Y_b)} \leq \ell \cdot 2^{\ell^2 3^{\delta^{-4}}} \leq 2^{\ell^3 3^{\delta^{-4}}}$. Applying the same argument to $ac$ and $bc$, we conclude that
	\[
		\ab{\p_E\up{\tau+\frac 12}(Y)} \leq 2^{\ab{\p_E\up \tau}^3 3^{\delta^{-4}}} \leq 2^{B(\tau)^3 3^{(t^2 B(\tau)/\eta)^{4\binom t2}}} \leq B(\tau+1).
	\]
	What remains now is to additionally refine the vertex cylinder partition $\p_V\up \tau$ in order to ensure that property \ref{it:edge quasirandom} holds for $\p\up{\tau+1}$. Fix some $Y \in \p_V\up \tau$. The edge parts in $\p_E\up{\tau+\frac 12}(Y)$ form a collection of $m \leq \binom t2 B(\tau+1)$ bipartite graphs, each with vertex set $Y(X_i) \cup Y(X_j)$ for some $i<j$. In particular, we may and will also view them as $t$-partite graphs\footnote{It would perhaps be more natural to apply the cylinder regularity lemma for \emph{bipartite} graphs (i.e.\ \cref{lem:graph cylinder} with $t=2$), regularizing each of the $\binom t2$ partitions $\p_E\up{\tau+\frac 12}(Y_i \times Y_j)$ one by one. The problem with this approach is that when we regularize a pair, we may destroy the regularity in previously-processed pairs. This is essentially the same reason why, if one wants to prove a regularity lemma for several graphs on the same vertex set, it is better to run the proof of the regularity lemma for all of them simultaneously, as opposed to applying the regularity lemma to each of them in turn.} on the vertex set $Y(X_1)\cup \dotsb \cup Y(X_t)$, which contain only edges between $Y(X_i)$ and $Y(X_j)$. Let $\varepsilon = (\eta/ (t^2B(\tau+1)))^{\binom t2}$, and let $\beta = \psi(\varepsilon)$; these will be the values of $\delta,\alpha$, respectively, corresponding to step $\tau+1$ of the iteration. Applying \cref{lem:graph cylinder}, we find a vertex cylinder partition of $Y$ into at most $2^{mt^2\beta^{-20}}$ parts, such that a $1-\beta/2$ fraction of the $t$-tuples of vertices define a $\beta$-quasirandom cylinder within each part of $\p_E\up{\tau+\frac 12}(Y)$. Let $\p_V\up{\tau+1}$ be the vertex cylinder partition obtained by partitioning each such $Y \in \p_V\up \tau$. Note that
	\[
		\ab{\p_V\up {\tau+1}} \leq 2^{\binom t2 B(\tau+1)\cdot t^2 \beta^{-20}} \ab{\p_V\up \tau} \leq A(\tau+1).
	\]
	Let $\p_E\up{\tau+1}$ be the edge partition obtained by restricting $\p_E\up{\tau+\frac 12}$ to the new vertex cylinders in $\p_V\up{\tau+1}$, and let $\p\up{\tau+1} = (\p_V\up{\tau+1},\p_E\up {\tau+1})$. Note that for every $Y \in \p_V\up {\tau+1}$, we have that $\ab{\p_E\up{\tau+1}(Y)} \leq\ab{\p_E\up {\tau+\frac 12}(\wh Y)} \leq B(\tau+1)$, where $\wh Y$ is the (unique) vertex cylinder in $\p_V\up \tau$ containing $Y$.  As $\p\up {\tau+1}$ is a refinement of $(\p_V\up \tau, \p_E\up{\tau+\frac 12})$, we have by \cref{lemit:P refinement} that
	\[
		q(\p\up{\tau+1}) \geq q(\p\up \tau) + 2^{-10}\eta^3/t^3 \geq 2^{-10} \eta^3 (\tau+1) /t^3,
	\]
	by applying property \ref{it:q LB} to $\p\up \tau$. This proves that properties \ref{it:PV UB}, \ref{it:PE UB}, and \ref{it:q LB} hold for $\p\up {\tau+1}$.

	Moreover, by our choice of $\p_V\up {\tau+1}$, we know that for every $Y \in \p_V\up {\tau+1}$, if we sample $(y_1,\dots,y_t) \in Y(X_1)\dtimes Y(X_t)$ uniformly at random, then with probability at least $1-\beta/2$, we have that $Z \coloneqq \p_E\up {\tau+1}[y_1,\dots,y_t]$ is $\beta$-quasirandom.
	Let us call the tuple $(y_1,\dots,y_t)$ \emph{sparse for $(i,j)$} if $d(Z_{ij})<\varepsilon^{1/\binom t2}$, or equivalently if the number of edges in $Z_{ij}$ is less than $\varepsilon^{1/\binom t2} \ab{Y(X_i)} \ab{Y(X_j)}$. As there are at most $B(\tau+1)$ choices for $Z_{ij}$, the total number of edges contributed by the sparse tuples is less than $\varepsilon^{1/\binom t2} B(\tau+1)\ab{Y(X_i)}\ab{Y(X_j)}=(\eta/t^2)\ab{Y(X_i)}\ab{Y(X_j)}$.

	We now recall that the graphs $Z_{ij}$ partition the complete bipartite graph $Y(X_i)\times Y(X_j)$. Since at most an $(\eta/t^2)$-fraction of those edges come from tuples that are sparse for $(i,j)$, we conclude that the probability that $(y_1,\dots,y_t)$ is sparse for $(i,j)$ is at most $\eta/t^2$. However, if $\delta(Z)< \varepsilon$ then $d(Z_{ij})<\varepsilon^{1/\binom t2}$ for some $i,j$, and hence the probability that $\delta(Z)<\varepsilon$ is at most $\binom t2 (\eta/t^2)$. In total, the probability of violating property \ref{it:edge quasirandom} is at most $\beta/2+\binom t2(\eta/t^2) \leq \eta/2$. Recall that this holds for any fixed $Y$, and for a uniformly random sample of $(y_1,\dots,y_t) \in Y(X_1)\dtimes Y(X_t)$; in particular, this implies that the same holds for uniformly random $(x_1,\dots,x_t) \in X_1 \dtimes X_t$, demonstrating that property \ref{it:edge quasirandom} holds for step $\tau+1$. This completes the proof of the inductive step.

	To obtain the desired final partition, we recall by \cref{lemit:P0t3} that $q(\p\up \tau) \leq \binom t3$ for all $\tau$. Therefore, by property \ref{it:q LB}, this process must terminate at some step $\tau \leq \binom t3/(2^{-10}\eta^3/t^3) \leq 10^4 t^6 \eta^{-3}$. At that step, by property \ref{it:stop}, we must have found a partition $\p=(\p_V,\p_E)$ such that $\h$ is $(\eta,\psi)$-quasirandom with respect to $\p$. Moreover, by \ref{it:edge quasirandom}, we have that a $(1-\eta)$-fraction of the $t$-tuples define a chain with product density at least $\delta$, where $\delta = (\eta/(t^2 B(10^4 t^6 \eta^{-3})))^{\binom t2}$ depends only on $t,\eta$, and $\psi$. Finally, from \eqref{eq:AB UB}, we have that
	\[
		\ab{\p_V} \leq A (10^4 t^6/\eta^3) \leq \twr(10^5 t^6/\eta^3, \psi(\eta/t^2)^{-20})
	\]
	and
	\[
		\ab{\p_E} \leq B(10^4 t^3/\eta^3) \leq \twr(10^5 t^6/\eta^3, (t^2/\eta)^{10}).\qedhere
	\]
\end{proof}

\section{Proof of Theorem \ref{thm:main UB informal}}\label{sec:UB proof}
In this section we prove \cref{thm:main UB formal}, following the proof sketch discussed in \cref{sec:proof sketch}. Our main tools will be \cref{thm:hyper cylinder with delta} from the previous section, as well as the induced counting lemma, which we now discuss.

The main property we need about $3$-graphs of bounded VC$_2$ dimension is that in every quasirandom chain, the relative density is very close to $0$ or $1$. This follows from the induced counting lemma (\cref{lem:induced counting} below), which informally states that if a tripartite $3$-graph $\h$ is sufficiently quasirandom, and the relative density of hyperedges is bounded away from $0$ and $1$, then it contains any fixed tripartite $3$-graph as an induced subhypergraph. For simplicity, we only state the result for finding one induced copy, but it is easy to see that one could similarly get a counting result.
It is well-known that such induced statements follow from general counting lemmas (see e.g.\ \cite[Remark 2.6.3]{MR4603631}), and indeed, our result follows directly from Gowers's hypergraph counting lemma \cite[Corollary 5.3]{MR2373376}.
Essentially the same statement, without explicit constants, is proved in \cite[Corollary 2.16]{2404.02024}; we include a proof in order to fully track the quantitative dependencies.

\begin{lemma}\label{lem:induced counting}
	Let $\V$ be a tripartite $3$-graph with $r$ vertices. Let $G$ be a tripartite graph on parts $U_1,U_2,U_3$, and let $\h_0$ be a tripartite $3$-graph supported on $\Delta(G)$.

	Let $\delta, \gamma\in (0, \frac 1{100}]$, and define $\eta=10^{-16}\gamma^{10r^3}$ and $\psi(x)=10^{-20}\gamma^{10r^3}x^{16r^2}$. Suppose that $\delta(G)\geq \delta$, that $$\gamma \leq d(\h_0 \mid G)\leq 1-\gamma,$$ and that $\h_0[U;G]$ is $(\eta,\psi)$-quasirandom. If $\min \{\ab{U_1},\ab{U_2},\ab{U_3}\}\geq \delta^{-r^2}\gamma^{-r^3}$, then $\V$ is an induced subhypergraph of $\h$.
\end{lemma}
\noindent As the proof of \cref{lem:induced counting} uses standard tools, we defer it to \cref{appendix:induced counting}.

Our next lemma is an immediate corollary of \cref{lem:induced counting}, and states that a sufficiently quasirandom $3$-graph of bounded VC$_2$ dimension must have relative density close to $0$ or $1$.
Results of this type are quite well-known (see e.g.\ \cite[Proposition 4.2]{2111.01737} or \cite[Proposition 3.4]{2404.02024}), but we include a proof in order to explicitly keep track of the quantitative dependencies.

Recall from the introduction that the tripartite $3$-graph $\V_d$ is defined as follows. Its vertex set is $A\sqcup B \sqcup C$, where $A = B = [d]$ and $C = 2^{[d]^2}$, and its hyperedges are all triples $(a,b,S) \in A \times B \times C$ with $(a,b) \in S$. 

\begin{lemma}\label{lem:all cylinders 01}
	Let $d \geq 1$ be an integer, and let $r=2(d+1)+2^{(d+1)^2}$. Let $G$ be a tripartite graph on parts $U_1,U_2,U_3$, and let $\h_0$ be a tripartite $3$-graph supported on $\Delta(G)$.

	Let $\delta, \gamma\in (0, \frac 1{100}]$, and define $\eta=10^{-16}\gamma^{10r^3}$ and $\psi(x)=10^{-20}\gamma^{10r^3}x^{16r^2}$. Suppose that $\delta(G)\geq \delta$, that $\h_0[U;G]$ is $(\eta,\psi)$-quasirandom, and that $\min \{\ab{U_1},\ab{U_2},\ab{U_3}\}\geq \delta^{-r^2}\gamma^{-r^3}$.

	If $\h_0$ has VC$_2$ dimension at most $d$, then
	\[
		d(\h_0\mid G) \in [0,\gamma] \cup [1-\gamma,1].
	\]
\end{lemma}
\begin{proof}
	Suppose for contradiction that $\gamma < d(\h_0\mid G)< 1-\gamma$. Then we are precisely in the setting of \cref{lem:induced counting}, which implies that the $r$-vertex tripartite $3$-graph $\V_{d+1}$ is an induced subhypergraph of $\h_0$. That is, the VC$_2$ dimension of $\h_0$ is at least $d+1$, a contradiction.
\end{proof}

Combining \cref{thm:hyper cylinder,lem:all cylinders 01}, we obtain the following result, which is a version of \cref{thm:main UB informal} for cylinder partitions.
\begin{lemma}\label{lem:cylinder partition 01}
	Let $d\geq 1, t \geq 3$ be integers and let $\gamma\in (0,\frac 1{100}]$ be a parameter.
	Let $\h$ be a sufficiently large $3$-graph with VC$_2$ dimension at most $d$, and let $X_1 \sqcup \dots \sqcup X_t$ be an arbitrary equitable partition of $V(\h)$. There is a cylinder chain partition $\p=(\p_V,\p_E)$ of $\h$ such that the following holds. If $(x_1,\dots,x_t) \in X_1 \dtimes X_t$ is sampled uniformly at random, then with probability at least $1-\gamma^2$, we have that
	\[
		d(\h[\p_V[x_1,\dots,x_t]; \p_E[x_1,\dots,x_t];\{i,j,k\}]) \in [0,\gamma] \cup [1-\gamma,1]
	\]
	for all $1\leq i<j<k\leq t$.
	Moreover,
	\[
		\ab{\p_V} \leq \twr \left( \left(\frac{t}{\gamma}\right)^C \right) \qquad \text{ and } \qquad \ab{\p_E} \leq \twr \left( \left(\frac{t}{\gamma}\right)^C \right),
	\]
	where $C>0$ is a constant depending only on $d$.
\end{lemma}
Note that in this statement, we are simply ignoring all edges of $\h$ which do not transverse the partition $X_1 \sqcup \dots \sqcup X_t$.
\begin{proof}[Proof of \cref{lem:cylinder partition 01}]
	Let $r = 2(d+1)+2^{(d+1)^2}$, let $\eta = 10^{-16}\gamma^{10r^3}$ and $\psi(x) = 10^{-20} \gamma^{10r^3}x^{16r^2}$. We apply \cref{thm:hyper cylinder with delta} to $\h$ to obtain a cylinder chain partition $\p$ such that $\h$ is $(\eta,\psi)$-quasirandom relative to $\p$. We also obtain some constant $\delta>0$ depending only on $t,\eta$ and $\psi$ such that a $(1-\eta)$-fraction of the tuples lie in a chain with product density at least $\delta$. Additionally, the claimed bounds on $\ab{\p_V}$ and $\ab{\p_E}$ follow immediately from \cref{thm:hyper cylinder with delta}.

	Now, as long as $\ab{V(\h)}$ is sufficiently large, we conclude that at least a $(1-3\eta)$-fraction of the tuples $(x_1,\dots,x_t)$ lie in a cylinder that is $(\eta,\psi)$-quasirandom, has product density at least $\delta$, and all $t$ of its vertex parts have size at least $\delta^{-r^2}\gamma^{-r^3}$. We now apply \cref{lem:all cylinders 01} to any triple of parts in this cylinder. The only thing to note is that this tripartite hypergraph has no tripartitely induced copy of $\V_{d+1}$, as such a copy is also tripartitely induced in $\h$ (recall that for tripartite inducedness, we only care about edges going between all three parts, and thus the additional edges $\h$ may have are irrelevant). We may thus apply \cref{lem:all cylinders 01}. We conclude that the relative density in this chain is at most $\gamma$ or at least $1-\gamma$. The claimed result now follows since $3\eta\leq \gamma^2$.
\end{proof}
With \cref{lem:cylinder partition 01} in hand, our goal is now to convert this cylinder chain partition into a genuine chain partition, while preserving the property that almost all chains have relative density close to $0$ or $1$. We do this following the outline presented in \cref{sec:our approach sketch}, and the remainder of the proof is dedicated to this step.

Our next lemma, which is a simple consequence of Markov's inequality, states that when we partition a quasirandom chain with relative density very close to $0$ or $1$, most subchains also have edge density close to $0$ or $1$.
\begin{lemma}\label{lem:arbitrary partition keeps 01}
	Let $G$ be a tripartite graph on parts $A,B,C$, and let $\h_0$ be a tripartite $3$-graph supported on $\Delta(G)$.
	Let $\gamma>0$ be a parameter.

	Suppose that $A,B,C$ are partitioned into $A_1,\dots,A_K, B_1,\dots,B_K,C_1,\dots,C_K$, respectively. Suppose too that each bipartite graph $G[A_i \times B_j], G[A_i \times C_j], G[B_i \times C_j]$ is further partitioned into $K$ subgraphs.

	Suppose that, among all triangles in $G$, a triangle $(x,y,z)$ is sampled uniformly at random. If $d(\h_0\mid G)<\gamma$, then with probability at least $1-\sqrt \gamma$, the unique subchain $G_\mu[A_i \times B_j] \cup G_\nu[A_i \times C_k] \cup G_\xi[B_j \times C_k]$ that contains $(x,y,z)$ satisfies
	\[
		d(\h_0 \mid G_\mu[A_i \times B_j] \cup G_\nu[A_i \times C_k] \cup G_\xi[B_j \times C_k]) < \sqrt \gamma.
	\]
	Similarly, if $d(\h_0\mid G)>1-\gamma$, then with probability at least $1-\sqrt \gamma$, we have
	\[
		d(\h_0\mid G_\mu[A_i \times B_j] \cup G_\nu[A_i \times C_k] \cup G_\xi[B_j \times C_k]) >1- \sqrt \gamma.
	\]
\end{lemma}
\begin{proof}
	Every triangle in $G$ is contained in precisely one subgraph of the form $(A_i \cup B_j \cup C_k, G_\mu \cup G_\nu \cup G_\xi)$. Therefore, these partitions of the vertex and edge sets
	induce a corresponding partition of $\Delta(G)$ into a collection of subsets, say $\Delta(G) = \Delta_1 \sqcup \dotsb \sqcup \Delta_M$. The uniform distribution on $\Delta(G)$ is the same as first picking an index $\ell \in [M]$ at random, with probability proportional to $\ab{\Delta_\ell}$, and then picking a uniformly random element of $\Delta_\ell$. Denote by $d_\ell$ the relative density $d(\h_0 \mid G_\mu[A_i \times B_j] \cup G_\nu[A_i \times C_k] \cup G_\xi[B_j \times C_k])$, where $A_i,B_j,C_k, G_\mu, G_\nu, G_\xi$ are the subchain defining $\Delta_\ell$. We then have that $d(\h_0\mid G)$ is equal to the weighted average of $d_\ell$, where the weight of $d_\ell$ is proportional to $\ab{\Delta_\ell}$; more precisely,
	\[
		d(\h_0\mid G) = \sum_{\ell=1}^M \frac{\ab{\Delta_\ell}}{\ab{\Delta(G)}} d_\ell.
	\]
	Said differently, in the random experiment above, the expected value of $d_\ell$ is exactly $d(\h_0\mid G)$, which we assumed is less than $\gamma$. By Markov's inequality, we conclude that
	\[
		\pr(d_\ell \geq \sqrt \gamma)\leq \frac{\E[d_\ell]}{\sqrt \gamma} < \frac{\gamma}{\sqrt \gamma} = \sqrt\gamma.
	\]
	This is precisely the first claimed statement.

	For the second statement, we simply replace $\h_0$ by its relative complement (namely the hypergraph whose hyperedges are precisely those triangles in $G$ which do not form a hyperedge of $\h_0$), and the result follows from the first statement.
\end{proof}

We will now begin working with chain partitions of hypergraphs, which are true chain partitions rather than partitions into cylinders. To minimize confusion, we use the letter $\Q$ for such objects, as in the introduction.

\begin{definition}
	Let $\p = (\p_V,\p_E)$ be a cylinder chain partition of a $t$-partite hypergraph $\h$ on vertex set $X_1 \sqcup \dotsb \sqcup X_t$. The \emph{Venn diagram partition} of $\p$ is the chain partition $\Q=\Q(\p)$ defined as follows. First, two vertices $x_i,x_i' \in X_i$ are in the same part of $\Q_V$ if and only if, for all $(x_1,\dots,x_{i-1},x_{i+1},\dots,x_t) \in X_1 \dtimes X_{i-1}\times X_{i+1} \dtimes X_t$, and for all $Y \in \p_V$, we have
	\[
		(x_1,\dots,x_{i-1},x_i,x_{i+1},\dots,x_t) \in Y \Longleftrightarrow (x_1,\dots,x_{i-1},x_i',x_{i+1},\dots,x_t) \in Y.
	\]
	That is, $x_i,x_i'$ are in the same part of $\Q_V$ if and only if they participate in exactly the same cylinders in $\p_V$.

	Now, given parts $Y_i \subseteq X_i, Y_j \subseteq X_j$ of $\Q_V$, we define the edge partition $\Q_E(Y_i \times Y_j)$ as follows. Two edges $(x_i,x_j), (x_i',x_j') \in Y_i \times Y_j$ lie in the same part of $\Q_E(Y_i \times Y_j)$ if and only if they participate in all of the same edge parts of $\p_E(Y)$, over all vertex cylinders $Y \in \p_V$ which contain $Y_i,Y_j$.
\end{definition}
Note that, if $\Q$ is the Venn diagram partition of $\p$, then $\ab{\Q_V} \leq t 2^{\ab{\p_V}}$, since, for each part $X_i$, there are at most $2^{\ab{\p_V}}$ possible intersection patterns of cylinders in $\p_V$ with $X_i$. Similarly, $\ab{\Q_E(Y_i \times Y_j)} \leq 2^{\ab{\p_E}\ab{\p_V}}$ for every $Y_i,Y_j \in \Q_V$, since there are at most $\ab{\p_V}$ options for a cylinder $Y \in \p_V$ containing $Y_i,Y_j$, and at most $\ab{\p_E}$ bipartite graphs in $\p_E(Y)$, and thus at most $2^{\ab{\p_E}\ab{\p_V}}$ possible intersection patterns with $Y_i \times Y_j$.

Our next lemma shows that by taking the Venn diagram partition of the cylinder chain partition given by \cref{lem:cylinder partition 01}, we can obtain a true chain partition such that almost all triples have relative density close to $0$ or $1$.
\begin{lemma}\label{lem:true partition 01}
	Let $d\geq 1, t \geq 3$ be integers and let $\zeta \in (0,\frac 1{10}]$ be a parameter.
	Let $\h$ be a sufficiently large $3$-graph with VC$_2$ dimension at most $d$, and let $X_1 \sqcup \dots \sqcup X_t$ be an arbitrary equitable partition of $V(\h)$. There is a chain partition $\Q=(\Q_V,\Q_E)$ of $\h$ such that the following holds for all $1 \leq i<j<k\leq t$. If $(x_i,x_j,x_k) \in X_i \times X_j \times X_k$ is sampled uniformly at random, then with probability at least $1-3\zeta$, we have that
	\[
		d(\h[\Q;x_i,x_j,x_k]) \in [0,\zeta] \cup [1-\zeta,1].
	\]
	Moreover,
	\[
		\ab{\Q_V} \leq \twr \left( \left(\frac{t}{\zeta}\right)^C \right) \qquad \text{ and } \qquad \ab{\Q_E(Y\times Y')} \leq \twr \left( \left(\frac{t}{\zeta}\right)^C \right)
	\]
	for every $Y,Y' \in \Q_V$, where $C>0$ is a constant depending only on $d$.
\end{lemma}
\begin{proof}
	Apply \cref{lem:cylinder partition 01} with parameter $\gamma=\zeta^2$ to obtain a cylinder chain partition $\p$, and let $\Q$ be the Venn diagram partition of $\p$. Note that the claimed bounds on $\ab{\Q_V},\ab{\Q_E(Y\times Y')}$ follow immediately from the corresponding bounds on $\ab{\p_V},\ab{\p_E}$.

	For the rest of the proof, we fix some $1 \leq i<j<k\leq t$. For a vertex cylinder $Y \in \p_V$, we define its \emph{homogeneous fraction} $h(Y)$ to be the probability, for uniformly random $(y_1,\dots,y_t) \in Y$, that $d(\h[\p_V[y_1,\dots,y_t];\p_E[y_1,\dots,y_t;\{i,j,k\}]]) \in [0,\gamma]\cup [1-\gamma,1]$. Equivalently, $h(Y)$ is the fraction of triples $(y_i,y_j,y_k) \in Y_i \times Y_j \times Y_k$ which define a tripartite chain whose relative density is in $[0,\gamma]\cup [1-\gamma,1]$, i.e.\
	\[
		h(Y) = \sum_{\substack{Z_{ij} \in \p_E(Y_i \times Y_j)\\Z_{ik} \in \p_E(Y_i \times Y_k)\\Z_{jk} \in \p_E(Y_j \times Y_k)}} \frac{\ab{\Delta(Z_{ij}\cup Z_{ik} \cup Z_{jk})}}{\ab{Y_i}\ab{Y_j}\ab{Y_k}} \bm 1_{d(\h \mid Z_{ij} \cup Z_{ik}\cup Z_{jk})\in [0,\gamma]\cup [1-\gamma,1]}.
	\] 
	Let us say that a vertex cylinder $Y \in \p_V$ is \emph{typical} if $h(Y) \geq 1-\gamma$. 
	
	To justify this name, we first claim that a $(1-\gamma)$-fraction of the tuples in $X_1\dtimes X_t$ lie in typical vertex cylinders. Indeed, \cref{lem:cylinder partition 01} states that if we sample $(x_1,\dots,x_t) \in X_1\dtimes X_t$ uniformly, then with probability at least $1-\gamma^2$, we have
	$d(\h[\p_V[x_1,\dots,x_t];\p_E[x_1,\dots,x_t;\{i,j,k\}]]) \in [0,\gamma]\cup [1-\gamma,1]$. But this probability is exactly the weighted average of $h(Y)$ over all cylinders $Y$, namely
	\[
		\sum_{Y \in \p_V} \left(\prod_{i=1}^t \frac{\ab{Y_i}}{\ab{X_i}}\right) h(Y) \geq 1-\gamma^2.
	\]
	By Markov's inequality, we find that if $(x_1,\dots,x_t) \in X_1\dtimes X_t$ is sampled uniformly at random, then with probability at least $1-\gamma$ it lies in a typical vertex cylinder. 

	Now, let $Y$ be a typical vertex cylinder.
	By the definition of typicality, a $(1-\gamma)$-fraction of the triples in $Y_i \times Y_j \times Y_k$ define edge parts $Z_{ij},Z_{ik},Z_{jk}$ for which the relative density is in $[0,\gamma]\cup [1-\gamma,1]$. Let us call such tuples $(Y_i \cup Y_j \cup Y_k, Z_{ij}\cup Z_{ik}\cup Z_{jk})$ \emph{doubly typical}. What we have now found is that, if $(x_1,\dots,x_t)$ is sampled uniformly at random, then it lies in a doubly typical chain with probability at least $1-2\gamma$. 

	Finally, we fix a doubly typical tuple $(Y_i \cup Y_j \cup Y_k, Z_{ij}\cup Z_{ik}\cup Z_{jk})$. By definition, $\Q_V$ and $\Q_E$ give some further partition of the vertex and edge set of this tripartite graph. By \cref{lem:arbitrary partition keeps 01}, we conclude that after this partitioning, a $(1-\sqrt \gamma)$-fraction of the triples of vertices in this tripartite graph still define tripartite subgraphs on which $\h$ has relative density in $[0,\sqrt \gamma]\cup [1-\sqrt \gamma,1]=[0,\zeta]\cup [1-\zeta,1]$. 

	Putting this all together, we find that if $(x_1,\dots,x_t) \in X_1 \dtimes X_t$ is sampled uniformly at random, the probability that $d(\h[\Q;x_i,x_j,x_k]) \in [0,\zeta] \cup [1-\zeta,1]$ is at least $1-2\gamma-\sqrt \gamma \geq 1-3\zeta$.
\end{proof}
The final result we need before proving \cref{thm:main UB formal} is a version of Szemer\'edi's regularity lemma for simultaneously regularizing multiple graphs, similar to \cref{lem:graph cylinder}, which is the cylinder regularity version of the same statement. This statement is well-known, e.g.\ \cite[Theorem 7.8]{MR2195580}.

Given a chain partition $(\Q_V,\Q_E)$ and a refinement $\Q_V'$ of $\Q_V$, we denote by $\Q_E'$ the edge partition obtained by restricting each bipartite graph in $\Q_E$ to the new vertex parts in $\Q_V$.
\begin{lemma}[{\cite[Theorem 7.8]{MR2195580}}]\label{lem:szemeredi for multiple graphs}
	Let $\alpha>0$, and let $\Q=(\Q_V,\Q_E)$ be a chain partition of some vertex set $V$. Let $T = \ab{\Q_V}$, and let $L$ be the maximum of $\ab{\Q_E(Y \times Y')}$ over all distinct $Y,Y' \in \Q_V$. There is a refinement $\Q_V'$ of $\Q_V$ such that a $(1-\alpha)$-fraction of pairs of vertices $(x_1,x_2)$ define an $\alpha$-quasirandom graph $\Q_E'[x_1,x_2]$. Moreover,
	\[
		\ab{\Q_V'} \leq \twr (16\alpha^{-3} T^2 L)
	\]
	and
	\[
		\ab{\Q_E'(Y\times Y')} \leq L
	\]
	for all $Y,Y' \in \Q_V'$.
\end{lemma}

With these tools, we are finally ready to prove \cref{thm:main UB formal}, our regularity lemma for hypergraphs of bounded VC$_2$ dimension.

\begin{proof}[Proof of \cref{thm:main UB formal}]
	Apply \cref{lem:true partition 01} with $\zeta=\eta^2/16$ and $t=3/\zeta$ to obtain a partition $\Q^{\hom}$. The choice of $t$ guarantees that at most a $\zeta$-fraction of the triples of vertices do not go between three different parts. Together with the conclusion of \cref{lem:true partition 01}, this implies that for a $(1-4\zeta)$-fraction of triples of vertices, the chain containing the triple has relative density in $[0,\zeta]\cup [1-\zeta,1]$. Let $L$ be the maximum of $\ab{\Q_E^{\hom}(Y \times Y')}$ over all $Y,Y' \in \Q_V^{\hom}$, so that $L \leq \twr(\eta^{-C_1})$ by \cref{lem:true partition 01}, for a constant $C_1$ depending only on $d$. Similarly, letting $T=\ab{\Q_V}$, we have that $T \leq \twr(\eta^{-C_1})$ by \cref{lem:true partition 01}.

	We now apply \cref{lem:szemeredi for multiple graphs} with parameter $\alpha = \psi((\eta/(9L))^3)$ to obtain a new chain partition $\Q=(\Q_V,\Q_E)$, obtained by refining $\Q_V^{\hom}$ and then restricting the graphs in $\Q_E^{\hom}$ to the new vertex parts. By \cref{lem:szemeredi for multiple graphs}, we find that
	\[
		\ab{\Q_V} \leq \twr (16\alpha^{-3} T^2 L) \leq \twr(\twr(\psi(\eta)^{-C_2}))
	\]
	for a constant $C_2$, and
	\[
		\ab{\Q_E(Y\times Y')} \leq L \leq \twr(\eta^{-C_1})
	\]
	for all $Y,Y' \in \Q_V$. By \cref{lem:szemeredi for multiple graphs}, a $(1-\alpha)$-fraction of pairs of vertices lie in $\alpha$-quasirandom bipartite graphs in $\Q_E$. Call such a bipartite graph \emph{sparse} if its edge density is less than $\eta/(9L)$. At most an $(\eta/9)$-fraction of pairs of vertices $(x,x')$ define a sparse graph $\Q_E[x,x']$. Therefore, if we pick a triple of vertices uniformly at random, then with probability at least $1-3(\eta/9 +\alpha)=1-\eta/3-3\alpha$, they lie in a chain of $\Q$ in which all of the three bipartite graphs are $\alpha$-quasirandom and not sparse. In particular, for every such chain, the corresponding tripartite graph $G$ satisfies $\delta(G) \geq (\eta/(9L))^3$, hence $G$ is $\psi(\delta(G))$-quasirandom by the monotonicity of $\psi$.

	Recall that in $\Q_V^{\hom}$, a $(1-4\zeta)$-fraction of triples lie in chains with relative density in $[0,\zeta]\cup [1-\zeta,1]$. By \cref{lem:arbitrary partition keeps 01}, we conclude that in $\Q_V$, a $(1-4\zeta-\sqrt{\zeta})$-fraction of triples lie in chains with relative density in $[0,\sqrt{\zeta}]\cup [1-\sqrt{\zeta},1]$. Thus, for a random triple, the failure probability is at most
	\[
		\frac \eta 3 + 3 \alpha + 4 \zeta + \sqrt \zeta \leq \eta
	\]
	using our assumption that $\psi(x) \leq x$ and our definitions of $\zeta$ and $\alpha$. This concludes the proof.
\end{proof}

\section{Proof of the lower bound}\label{sec:LB proof}
As mentioned in the introduction, \cref{thm:main LB informal} follows from the following two results of Terry. The first states, in a quantitative form, the fact that quasirandomness for $3$-graphs implies weak quasirandomness. 
This fact is well-known and not hard to prove, see e.g.\ \cite{MR2864650,MR2595699}. However, to the best of our knowledge, \cite[Lemma 5.15]{2404.02024} is the only place in the literature where this reduction is given a quantitative form.
\begin{proposition}[{Terry \cite[Lemma 5.15]{2404.02024}}]\label{prop:strong implies weak}
	There exists an absolute constant $C>0$ such that the following holds for all $\eta>0$ and all $\psi:(0,1) \to (0,1)$ satisfying $\psi(x) \leq x^C$. If a  partition $\Q=(\Q_V,\Q_E)$ of a 3-graph is $(\eta,\psi)$-quasirandom, then the partition $\Q_V$ is weakly $\eta^{1/C}$-quasirandom.
\end{proposition}
Second, we need the following result, which states that there exist $3$-graphs of bounded VC$_2$ dimension admitting no weakly quasirandom partitions with fewer than tower-type many parts.
\begin{proposition}[{Terry \cite[Theorem 5.4]{2404.01293}}]\label{prop:terry LB}
	For every sufficiently small $\eta>0$, there are infinitely many $3$-graphs $\h$ with the following properties. $\h$ has VC$_2$ dimension at most $1$, and if $\Q_V$ is a weakly $\eta$-quasirandom partition of $V(\h)$, then $\ab{\Q_V} \geq \twr(\poly(1/\eta))$.
\end{proposition}
Note that this lower bound is essentially best possible, since \cref{thm:chung} produces a weakly $\eta$-quasirandom partition with $\twr(\poly(1/\eta))$ parts for \emph{any} $\h$. 
Note too that by combining \cref{prop:strong implies weak,prop:terry LB}, we immediately obtain \cref{thm:main LB informal}.

For completeness, we briefly sketch the proof of \cref{prop:terry LB}, in order to explain where the tower-type bound comes from.
The construction used in the proof is the following simple operation, which turns a bipartite graph into a tripartite $3$-graph. 
\begin{definition}\label{def:n-fold link}
	Let $G$ be a bipartite graph with parts $A,B$, and let $n$ be an integer. The \emph{$n$-fold cone hypergraph} of $G$, denoted $n \otimes G$, is the $3$-graph with parts $A,B,C$, where $\ab C = n$, and with edge set
	\[
		E(n \otimes G) = \{(a,b,c) : ab \in E(G), c \in C\}.
	\]
\end{definition}
In other words, $n \otimes G$ is obtained from $G$ by adding $n$ new vertices, each of which has link precisely equal to $G$. The basic observation about this construction is that $n \otimes G$ has bounded VC$_2$ dimension, regardless of the choice of $G$.
\begin{lemma}\label{lem:link bounded VC2}
	For every bipartite graph $G$ and every integer $n$, the $3$-graph $n \otimes G$ has VC$_2$ dimension at most $1$.
\end{lemma}
\begin{proof}
	It suffices to prove that $n\otimes G$ has no shattered $K_{2,2}$. Let the parts of $n \otimes G$ be $A,B,C$, as in \cref{def:n-fold link}. As $n \otimes G$ is tripartite, any shattered $K_{2,2}$ must have its two parts contained in two of the parts $A,B,C$.

	Consider first some $a,a' \in A, b,b' \in B$. They clearly cannot form a shattered $K_{2,2}$, since all $c \in C$ have the same link when restricted to these four vertices, whereas in a shattered $K_{2,2}$ we would need to see $16$ different links on these four vertices.

	Next consider some $a,a' \in A, c,c' \in C$. We claim that for every $b \in B$, if $a c \in N(b)$, then $a c' \in N(b)$ as well. Note that this immediately implies that that $a,a',c,c'$ do not form a shattered $K_{2,2}$, as for example we can never see the subgraph $\{(a,c)\} \subseteq \{a,a'\}\times \{c,c'\}$ in the link of some $b \in B$. To prove the claim, suppose that $ac \in N(b)$ for some $b \in B$. By definition of $n \otimes G$, this implies that $ab \in E(G)$, hence that $abc' \in E(n \otimes G)$, and therefore that $ac' \in N(b)$ as well, as claimed.

	The final case, of a shattered $K_{2,2}$ among $B \times C$, is ruled out identically to the previous case, by simply interchanging the roles of $A$ and $B$.
\end{proof}
With this construction in hand, it is straightforward to describe the proof of \cref{prop:terry LB}. We begin by letting $G$ be Gowers's example of a graph admitting no $\eta$-quasirandom partition with fewer than $\twr(\poly(1/\eta))$ parts. We may moreover assume that $G$ is bipartite\footnote{Gowers's \cite{MR1445389} construction does not explicitly provide a balanced bipartite graph $G$, but it is easy to see that his construction can be modified to yield this, and this is also explicitly the case in the construction of Fox--Lov\'asz \cite{MR3737374}.} with both parts of size $n$. Now, we let $\h = n \otimes G$. By \cref{lem:link bounded VC2}, we know that $\h$ has bounded VC$_2$ dimension.

To complete the proof, it suffices to show that $\h$ has no weakly $\eta$-quasirandom partition into fewer than $\twr(\poly(1/\eta))$ parts. While proving this is somewhat long and tedious (see \cite[Lemma 5.3]{2404.01293} for details), the fact that such a statement is true should hopefully not be surprising: 
the $3$-graph $n \otimes G$ contains ``no extra information'' beyond the graph $G$, hence a weakly quasirandom partition of the former should immediately yield a quasirandom partition of the latter. Once one formalizes this intuitive argument, the proof of \cref{prop:terry LB} is complete.

\section{Concluding remarks and open problems}\label{sec:conclusion}

\subsection{Other applications of cylinder regularity for hypergraphs}\label{sec:cylinder apps}
The major technical tool introduced in this paper is the cylinder regularity lemma for hypergraphs, \cref{thm:hyper cylinder}, and we expect this result to have many other applications. We now briefly discuss two such applications, as well as some open problems related to them.

The first of these applications is a hypergraph extension of the quasirandom subset lemma of Conlon and Fox \cite{MR2989432}. For a graph $G$ and a set $U \subseteq V(G)$, we say that $U$ is \emph{$\varepsilon$-quasirandom} if the pair $(U,U)$ is $\varepsilon$-quasirandom. More formally, $U$ is $\varepsilon$-quasirandom if the pair $(U,U)$ is $\varepsilon$-quasirandom in the graph $G \times K_2$ (the \emph{bipartite double cover} of $G$), whose vertex set is two copies of $V(G)$, and whose edges are all pairs crossing the two parts that correspond to an edge of $G$. The quasirandom subset lemma then states that every graph contains a linear-sized quasirandom subset.
\begin{lemma}[{\cite[Lemma 5.2]{MR2989432}}]\label{lem:conlon fox}
	For every $\varepsilon>0$, there exists $\delta>0$ such that for every graph $G$, there is a set $U \subseteq V(G)$ with $\ab U \geq \delta \ab{V(G)}$ such that $U$ is $\varepsilon$-quasirandom. 
\end{lemma}
\cref{lem:conlon fox} is extremely useful (see e.g.\ \cite{MR4482092,MR3784260} and \cite[Chapter 6]{yuvalNotes} for several applications). This statement can be deduced from Szemer\'edi's regularity lemma, but the remarkable insight of Conlon and Fox is that one can find an alternative proof which yields much stronger quantitative bounds; Conlon and Fox proved that $1/\delta$ can be taken to be merely double-exponential in $1/{\poly(\varepsilon)}$. The proof of \cref{lem:conlon fox} with such a bound is quite short: it follows from an application of the cylinder regularity lemma and Ramsey's theorem (although see \cite[Lemma 6.1.4]{yuvalNotes} for a direct energy-increment proof).

One of the many immediate consequences of \cref{lem:conlon fox} is a famous result of R\"odl \cite{MR0837962}, which we now discuss. 
Recall that a graph $G$ is said to be \emph{induced-$F$-free} if it does not have an induced subgraph isomorphic to $F$. R\"odl proved that if $G$ is induced-$F$-free, it has a linear-sized vertex subset whose edge density is close to $0$ or $1$. More precisely, there is a set $U \subseteq V(G)$ with $d_G(U) \in [0,\varepsilon] \cup [1-\varepsilon,1]$, where $\ab U \geq \delta \ab{V(G)}$ and $\delta>0$ depends only on $F$ and $\varepsilon$. R\"odl's theorem turns out to be closely connected to other fundamental questions in graph theory such as the Erd\H os--Hajnal conjecture, see e.g.\ \cite{2403.08303} for details. 

There are a number of ways known for proving R\"odl's theorem, but one of the simplest is as a direct consequence of \cref{lem:conlon fox}: by the induced counting lemma, a quasirandom set in an induced-$F$-free cannot have edge density bounded away from $0$ or $1$. See e.g.\ \cite[Theorem 6.3.3]{yuvalNotes} for a detailed proof via this method.

As R\"odl himself observed \cite{MR0837962}, the obvious extension of this statement to hypergraphs is false. Indeed, the well-known \emph{random tournament} construction---the $3$-graph whose hyperedges are the cyclic triangles in a random tournament---is induced-$K_4\up 3$-free, but all of its linear-sized subsets have edge density $\frac 14 +o(1)$. In fact, the largest subsets of this $3$-graph with density close to $0$ or $1$ have logarithmic order, and a result of Amir, Shapira, and Tyomkyn \cite{MR4177961} shows that this is the smallest possible in an induced-$F$-free $3$-graph.

The same random tournament construction is one of the classical examples demonstrating the difficulty in defining a successful notion of hypergraph regularity. Indeed, the same construction demonstrates that there is no direct hypergraph analogue of \cref{lem:conlon fox}, in the sense that every linear-sized vertex subset of this random tournament hypergraph spans a $3$-graph which is not quasirandom. As such, one can learn the lessons of hypergraph regularity, and decide that in order to sensibly extend R\"odl's theorem to hypergraphs, one should not only restrict the vertex set of $\h$, but also the set of pairs. Indeed, we prove the following theorem.
\begin{theorem}\label{thm:hypergraph rodl}
	Let $F$ be a $3$-graph, let $\varepsilon>0$, and let $\psi:(0,1) \to (0,1)$ be a polynomial function. There exists $\delta>0$ such that the following holds. If $\h$ is an induced-$F$-free $3$-graph, then there is a subset $U \subseteq V(\h)$ and a graph $G$ with vertex set $U$ such that $\ab U \geq \delta \ab{V(\h)}$, $e(G) \geq \delta \ab U^2$, $G$ is $\psi(\delta)$-quasirandom, and
	\[
		d(\h[U,G]\mid G) \in [0,\varepsilon] \cup [1-\varepsilon,1].
	\]
	Moreover, we may take $1/\delta \leq \twr(2^{2^{\poly(1/\varepsilon)}})$.
\end{theorem}
Note that it is natural to insist that $G$ is highly quasirandom, as otherwise the result may be vacuously true (e.g.\ if $G$ is a complete bipartite graph, which has positive density but no triangles). Moreover, this can be viewed as a direct extension of the condition in R\"odl's theorem, since every set of vertices is ``highly quasirandom in uniformity $1$''. 

We prove \cref{thm:hypergraph rodl} via the same proof technique discussed above, and as such, we begin by proving an analogue of \cref{lem:conlon fox} for $3$-graphs.
Given a hypergraph $\h$, a vertex subset $U \subseteq V(\h)$, and a graph $G$ with vertex set $U$, let $\h[U,G]$ denote the subhypergraph of $\h$ obtained by only keeping the vertices in $U$, and only keeping the hyperedges supported on $\Delta(G)$. Let us say that $\h[U,G]$ is $(\eta,\psi)$-quasirandom if the following tripartite chain is $(\eta,\psi)$-quasirandom: its vertex set consists of three copies of $U$, its tripartite graph consists of three copies of the bipartite double cover of $G$, and its hyperedges are all triples across the three parts corresponding to hyperedges in $\h$. With this terminology, we can state the following hypergraph extension of \cref{lem:conlon fox}.
\begin{theorem}\label{thm:hypergraph quasirandom subset}
	For every $\eta>0$ and every polynomial function $\psi:(0,1) \to (0,1)$, there exists $\delta>0$ such that the following holds for every $3$-graph $\h$. There is a subset $U \subseteq V(\h)$ and a graph $G$ with vertex set $U$ such that $\ab U \geq \delta \ab{V(\h)}$, $e(G) \geq \delta \ab U^2$, and $\h[U,G]$ is $(\eta,\psi)$-quasirandom. 

	Moreover, we may take $1/\delta \leq \twr(2^{2^{\poly(1/\eta)}})$.
\end{theorem}
In other words, we may restrict $V(\h)$ to a constant-density subset, and restrict the complete graph on $U$ to a constant-density subgraph, such that the hypergraph $\h[U,G]$ supported on these substructures is quasirandom. Note that, as above, we must include the graph $G$ for such a statement to be true: the random tournament construction witnesses that no such statement can hold if we are not allowed to restrict both the set of vertices and the set of pairs. The proof of \cref{thm:hypergraph quasirandom subset} closely follows that of \cref{lem:conlon fox}, with the main difference being an application of \cref{thm:hyper cylinder} rather than the cylinder regularity lemma for graphs. Note that, as in the graph case, one can easily deduce a statement like \cref{thm:hypergraph quasirandom subset} from the hypergraph regularity lemma, but such a proof would yield wowzer-type bounds; the main upshot of \cref{thm:hypergraph quasirandom subset} is the quantitative improvement over this.
\begin{proof}[Proof sketch of \cref{thm:hypergraph quasirandom subset}]
	Let $t$ be a large parameter to be chosen later, and arbitrarily partition $V(\h)$ into $t$ equal-sized parts. Apply \cref{thm:hyper cylinder} to obtain an $(\eta^2,\psi^2)$-quasirandom cylinder chain partition of this $t$-partite $3$-graph. In fact, we will not need the whole partition, and we may simply restrict to a single quasirandom cylinder $(Y,Z)$. Let us assume for simplicity that the vertex cylinder is equitable, i.e.\ that $\ab{Y_1}=\dots=\ab{Y_t}$.

	By randomly subsampling, we may assume that all bipartite graphs appearing in $Z$ have the same edge density $\delta$; it is well-known (and follows straightforwardly\footnote{For this proof, it is very convenient to use the definition of quasirandomness in \cref{def:graph quasirandomness}, rather than other equivalent notions like $\varepsilon$-regularity, for which naive applications of Azuma's inequality and the union bound do not give sufficiently strong estimates.} from Azuma's inequality) that random subsampling preserves quasirandomness, up to some small losses. We now define an edge-coloring of the complete $t$-vertex $3$-graph with $1/\eta^2$ colors, as follows. The triple $\{i,j,k\}$ receives color $\ell$ if $d(\h[Y;Z;\{i,j,k\}]) \in [(\ell-1)\eta^2,\ell\eta^2]$ (if a triple receives two colors, we break ties arbitrarily). By Ramsey's theorem for hypergraphs, we may pass to a subset of $s$ parts such that all triples among these parts receive the same color. 

	By renaming the indices, this means we have found disjoint vertex subsets $Y_1,\dots,Y_{s}$ as well as bipartite graphs $Z_{ij}\subseteq Y_i \times Y_j$, all of the same edge density, such that for all $1 \leq i<j<k\leq s$, all the chains $\h[Y;Z;\{i,j,k\}]$ are $(\eta^2,\psi^2)$-quasirandom, and have the same relative density up to an additive error of $\eta^2$. At this point, we let $U = Y_1 \cup \dots \cup Y_s$,  and let $E(G) = \bigcup_{i\leq j}E(Z_{ij})$, where we define $Z_{ii}$ to be a random graph with vertex set $Y_i$ and edge density $\delta$ for all $i$. In particular, with high probability as $\ab{V(\h)} \to \infty$, each graph $Z_{ii}$ is $\psi(\delta)^2$-quasirandom. Thus $G$ is the union of graphs, each of which has edge density $\delta$ and is $\psi(\delta)^2$-quasirandom, which implies that $G$ itself has edge density $\delta$ and is $\psi(\delta)$-quasirandom.
	
	Finally, we claim that if $s \geq 10/\eta^2$, then $\h[U,G]$ is $(\eta,\psi)$-quasirandom. Indeed, if $s$ is so large, then at least a $(1-\eta^2)$-fraction of triples in $U^3$ lie in three distinct parts $Y_i,Y_j,Y_k$. As all chains $\h[Y;Z;\{i,j,k\}]$ are $(\eta^2,\psi^2)$-quasirandom and have the same density up to an additive error of $\eta^2$, it is not hard to conclude that $\h[U,G]$ is indeed $(\eta,\psi)$-quasirandom (see \cite[Lemma 5.5]{MR2989432} for the analogous argument in uniformity $2$).

	This concludes the proof sketch; all that remains is to understand its quantitative aspects. We set $s= 10/\eta^2$. We then need $t$ to be at least the $(1/\eta^2)$-color hypergraph Ramsey number of $s$, so we may take $t=2^{2^{100\eta^{-4}}}$. Finally, \cref{thm:hyper cylinder} then implies that both $\delta$ and $\ab{Y_i}/\ab{V(\h)}$ can be lower-bounded as $1/\twr(2^{2^{\poly(1/\eta)}})$, as claimed.
\end{proof}

It is now straightforward to deduce our generalization of R\"odl's theorem, \cref{thm:hypergraph rodl}, from \cref{thm:hypergraph quasirandom subset}.

\begin{proof}[Proof sketch of \cref{thm:hypergraph rodl}]
	Let $\eta$ be polynomially small with respect to $\varepsilon$. By potentially replacing $\psi$ with a smaller polynomial function, we may assume that the induced counting lemma (\cref{lem:induced counting}) applies, that is, that whenever $\h[U,G]$ is $(\eta,\psi)$-quasirandom with relative density in $[\varepsilon,1-\varepsilon]$, then $\h[U,G]$ contains an induced copy of $F$. In fact, by applying the induced counting lemma for chains \cite[Corollary 5.3]{MR2373376}, we may assume that in this situation, we can find an induced copy of $F$ so that all pairs of vertices in $F$ span an edge of $G$. 
	
	We now apply \cref{thm:hypergraph quasirandom subset} to $\h$ to obtain $U$ and $G$ as above.We are done if $d(\h[U,G]\mid G) \in [0,\varepsilon] \cup [1-\varepsilon,1]$. If this does not happen, then we can find an induced copy of $F$ in $\h[U,G]$ as above. In particular, the fact that all pairs of vertices in $F$ span an edge of $G$ implies that this copy is also induced in $\h$, a contradiction.
\end{proof}
The proofs above give tower-type bounds in \cref{thm:hypergraph quasirandom subset,thm:hypergraph rodl}, and we think it would be extremely interesting to determine whether these bounds are close to best possible.
\begin{question}\label{qu:subtower for applications}
	Can one prove \cref{thm:hypergraph quasirandom subset} or \cref{thm:hypergraph rodl} with bounds of constant tower height? 
\end{question}
The analogous questions in uniformity $2$ are also very interesting:
in \cref{lem:conlon fox}, it remains a major open problem to determine whether the double-exponential dependence can be improved to single exponential (which would be best possible). In R\"odl's theorem, quasipolynomial bounds are known \cite{MR2455625}, and improving these bounds to polynomial is now known to be equivalent to proving the Erd\H os--Hajnal conjecture \cite{2403.08303}.

In order to answer \cref{qu:subtower for applications}, it is natural to try to improve the bounds in the cylinder regularity lemma. 
However, we believe that this is impossible, and expect the bound in \cref{thm:hyper cylinder} to be essentially best possible.
\begin{conjecture}\label{conj:tower for cylinder}
	There exists a tripartite $3$-graph $\h$ for which every $(\eta,\psi)$-quasirandom cylinder partition has at least $\twr(\psi(\eta)^{-C})$ parts, for some constant $C$.
\end{conjecture}
It may be possible to prove \cref{conj:tower for cylinder} by mimicking the technique of Moshkovitz and Shapira \cite{MR4025519}. Alternately, it may be possible to deduce \cref{conj:tower for cylinder} as a black-box reduction to the Moshkovitz--Shapira theorem: it may be possible to prove the full hypergraph regularity lemma by iterative applications of \cref{thm:hyper cylinder}, in which case sub-tower bounds for \cref{thm:hyper cylinder} would yield sub-wowzer bounds for the $3$-graph regularity lemma, contradicting the lower bound of Moshkovitz--Shapira. In uniformity $2$, it is well-known (see \cite[Section 2.2]{MR2798368} or \cite[Section 5.4]{MR2989432}) that iterative applications of weak regularity lemmas such as those of Duke--Lefmann--R\"odl \cite{MR1333857} or Frieze--Kannan \cite{MR1723039} can be used to prove Szemer\'edi's regularity lemma, hence we expect a similar technique to exist in higher uniformities as well.

\subsection{Further open problems}
Another natural problem left open by our results is to close the gap between \cref{thm:main LB informal,thm:main UB formal}, i.e.\ to determine whether the true behavior for $3$-graphs with bounded VC$_2$ dimension is of tower type or of double-tower type. We conjecture that the lower bound is closer to the truth.
\begin{conjecture}\label{conj:tower vs double tower}
	The bound in \cref{thm:main UB formal} can be improved to $\ab{\Q_V} \leq \twr(\psi(\eta)^{-C})$.
\end{conjecture}

We recall that throughout this paper, we have worked with Gowers's notion of hypergraph regularity, because this notion is known to support a counting lemma with polynomial bounds (concretely, we required that in \cref{lem:induced counting,lem:all cylinders 01}, $\eta$ and $\psi$ depend polynomially on $\gamma$, and that $\psi$ itself is a polynomial). As far as we are aware, all known proofs of the counting lemma for the regularity notion of R\"odl et al.\ use, as an intermediate step of the process, an application of the hypergraph regularity lemma. As such, the counting lemma requires terrible bounds. However, it is not clear if this is the truth, or merely an artifact of the known proofs.
\begin{question}
	Can one prove a counting lemma with polynomial bounds for other notions of hypergraph regularity?
\end{question}

Finally, we remark that \cref{thm:main UB informal} gives a rough structural characterization of all $3$-graphs with bounded VC$_2$ dimension: in some sense, the only way to obtain such a $3$-graph is to start with some graph, and include all of its triangles as hyperedges. More precisely, \cref{thm:main UB informal} implies that if $\h$ has bounded VC$_2$ dimension, then it is constructed in such a way from a bounded number of graphs, apart from a small amount of ``noise''; this is precisely the meaning of the fact that $\h$ has a regularity partition in which almost all triples have relative density close to $0$ or $1$. 

From the perspective of such a structural characterization, it is no longer necessary to demand that the graphs arising in the partition are themselves highly quasirandom (recall that in the setting of hypergraph regularity, and specifically the counting lemma, such an assumption is absolutely crucial). To make this notion precise, let us assume for simplicity that $\h$ is tripartite with parts $X,Y,Z$, each of size $n$. We then define a \emph{$\varepsilon$-graph partition} of $\h$ to be partitions $X\times Y = E_1 \sqcup \dots \sqcup E_M, X \times Z = F_1 \sqcup \dots \sqcup F_M, Y \times Z = G_1 \sqcup \dots \sqcup G_M$, as well as a function $f:[M]^3 \to \{0,1\}$, with the property that for all but $\varepsilon n^3$ choices of $(x,y,z) \in X \times Y \times Z$, we have
\[
	xyz \in E(\h) \Longleftrightarrow f(i,j,k)=1, \qquad \text{where }xy \in E_i, xz \in F_j, yz \in G_k.
\]
\cref{thm:main UB informal} implies that if $\h$ has bounded VC$_2$ dimension, then it admits an $\varepsilon$-graph partition with $M$ bounded by a tower-type function. Note that although the results of Terry \cite{MR4662634,2404.02030} imply polynomial bounds on the number of edge parts in a regularity decomposition, they only give wowzer-type bounds on the size of an $\varepsilon$-graph partition, because they still require a wowzer-type number of vertex parts, and hence a wowzer-type number of graphs in such a partition. However, we conjecture that the bound can be improved to polynomial.
\begin{conjecture}\label{conj:graph partition}
	If $\h$ has bounded VC$_2$ dimension, then it has an $\varepsilon$-graph partition with $M =\poly(1/\varepsilon)$ many graphs. 
\end{conjecture}
Note that, if true, \cref{conj:graph partition} would imply \cref{conj:tower vs double tower}, as we could simply apply Szemer\'edi's regularity lemma to all the graphs in such an $\varepsilon$-graph partition, while maintaining the homogeneity property by \cref{lem:arbitrary partition keeps 01}.


\appendix
\section{Proofs of technical lemmas}

\subsection{Proof of Lemma \ref{lem:q properties}}\label{appendix:mean-squared}
\begin{proof}[Proof of \cref{lem:q properties}]
	For \cref{lemit:G01}, we first note that
	the definition \eqref{eq:mean-squared} of $q(\p_E(G))$ is a sum of non-negative numbers, hence is $q(\p_E(G))$ is certainly non-negative. For the upper bound, we have that $d(\h \mid Z_{ij} \cup Z_{ik} \cup Z_{jk}) \leq 1$ for all $Z_{ij},Z_{ik},Z_{jk}$, and therefore
	\begin{align*}
		q(\p_E(G)) &= \sum_{\substack{Z_{ij} \in \p_E(G;Y_i \times Y_j)\\Z_{ik} \in \p_E(G;Y_i \times Y_k)\\Z_{jk} \in \p_E(G;Y_j \times Y_k)}} \frac{\ab{\Delta(Z_{ij}\cup Z_{ik}\cup Z_{jk})}}{\ab{\Delta(G)}} d(\h \mid Z_{ij}\cup Z_{ik}\cup Z_{jk})^2\\
		&\leq \sum_{\substack{Z_{ij} \in \p_E(G;Y_i \times Y_j)\\Z_{ik} \in \p_E(G;Y_i \times Y_k)\\Z_{jk} \in \p_E(G;Y_j \times Y_k)}} \frac{\ab{\Delta(Z_{ij}\cup Z_{ik}\cup Z_{jk})}}{\ab{\Delta(G)}}\\
		&=1,
	\end{align*}
	as every triangle in $G$ appears in exactly one tripartite subgraph $Z_{ij} \cup Z_{ik} \cup Z_{jk}$.
	This proves \cref{lemit:G01}

	For \cref{lemit:P0t3}, we argue similarly. First, the lower bound $q(\p) \geq 0$ is again immediate as $q(\p)$ is defined as a sum of non-negative numbers. For the upper bound, we have
	\begin{align*}
		q(\p) &= \sum_{Y \in \p_V} \left( \prod_{i=1}^t \frac{\ab{Y_i}}{\ab{X_i}} \right) \sum_{1 \leq i<j<k\leq t} q(\p_E(K(Y_i,Y_j,Y_k)))\\
		&\leq \sum_{Y \in \p_V} \left( \prod_{i=1}^t \frac{\ab{Y_i}}{\ab{X_i}} \right) \sum_{1 \leq i<j<k\leq t} 1\\
		&= \binom t3 \sum_{Y \in \p_V} \left( \prod_{i=1}^t \frac{\ab{Y_i}}{\ab{X_i}} \right)\\
		&= \binom t3,
	\end{align*}
	where the first inequality uses \cref{lemit:G01} applied to the complete tripartite graph $K(Y_i,Y_j,Y_k)$, and the final equality uses the fact that $\p_V$ is a vertex cylinder partition of $X_1\dtimes X_t$.

	It remains to prove \cref{lemit:G refinement,lemit:P refinement}, namely the monotonicity of the mean-squared density with respect to refinements. Both of these are consequences of the Cauchy--Schwarz inequality, as follows.

	Let $(G,\h)$ be a tripartite chain with vertex set $Y_1 \cup Y_2 \cup Y_3$, and let $\p_E(G),\p_E'(G)$ be two edge partitions of it with $\p_E'(G)$ refining $\p_E(G)$. Fix some $Z_{12} \in \p_E(G;Y_1\times Y_2), Z_{13} \in \p_E(G;Y_1 \times Y_3), Z_{23} \in \p_E(G;Y_2 \times Y_3)$. As $\p_E'(G)$ is a refinement of $\p_E(G)$, in $\p_E'(G)$ each of these bipartite graphs is further partitioned, say into $Z_{12}^1,\dots,Z_{12}^m, Z_{13}^1,\dots,Z_{13}^m, Z_{23}^1,\dots,Z_{23}^m$; we may assume that the number of parts in all three cases equals $m$, by potentially letting some of these bipartite graphs be empty. We now have that
	\begin{align*}
		d(\h\mid Z_{12}\cup Z_{13}\cup Z_{23})&=\frac{\ab{E(\h)\cap \Delta(Z_{12}\cup Z_{13}\cup Z_{23})}}{\ab{\Delta(Z_{12}\cup Z_{13}\cup Z_{23})}}\\
		&=\sum_{a,b,c=1}^m \frac{\ab{E(\h\cap \Delta(Z_{12}^a\cup Z_{13}^b\cup Z_{23}^c))}}{\ab{\Delta(Z_{12}\cup Z_{13}\cup Z_{23})}}
		\\
		&=\sum_{a,b,c=1}^m \frac{\ab{\Delta(Z_{12}^a\cup Z_{13}^b\cup Z_{23}^c)}}{\ab{\Delta(Z_{12}\cup Z_{13}\cup Z_{23})}} d(\h \mid Z_{12}^a \cup Z_{13}^b \cup Z_{23}^c),
	\end{align*}
	where the final step uses the definition of the relative density $d(\h \mid Z_{12}^a \cup Z_{13}^b \cup Z_{23}^c)$.
	Therefore, by Cauchy--Schwarz,
	\begin{align*}
		d(\h\mid& Z_{12}\cup Z_{13}\cup Z_{23})^2 =\left( \sum_{a,b,c=1}^m \frac{\ab{\Delta(Z_{12}^a\cup Z_{13}^b\cup Z_{23}^c)}}{\ab{\Delta(Z_{12}\cup Z_{13}\cup Z_{23})}} d(\h \mid Z_{12}^a \cup Z_{13}^b \cup Z_{23}^c) \right)^2\\
		&\leq \left( \sum_{a,b,c=1}^m \frac{\ab{\Delta(Z_{12}^a\cup Z_{13}^b\cup Z_{23}^c)}}{\ab{\Delta(Z_{12}\cup Z_{13}\cup Z_{23})}} \right) \left( \sum_{a,b,c=1}^m \frac{\ab{\Delta(Z_{12}^a\cup Z_{13}^b\cup Z_{23}^c)}}{\ab{\Delta(Z_{12}\cup Z_{13}\cup Z_{23})}} d(\h \mid Z_{12}^a \cup Z_{13}^b \cup Z_{23}^c)^2 \right)\\
		&=\sum_{a,b,c=1}^m \frac{\ab{\Delta(Z_{12}^a\cup Z_{13}^b\cup Z_{23}^c)}}{\ab{\Delta(Z_{12}\cup Z_{13}\cup Z_{23})}} d(\h \mid Z_{12}^a \cup Z_{13}^b \cup Z_{23}^c)^2,
	\end{align*}
	where in the final step we use that every triangle in $\Delta(Z_{12}\cup Z_{13}\cup Z_{23})$ appears in exactly one subgraph $Z_{12}^a\cup Z_{13}^b\cup Z_{23}^c$. Plugging this inequality into the definition \eqref{eq:mean-squared}, we find that
	\begin{align*}
		q(\p_E&(G)) = \sum_{\substack{Z_{12} \in \p_E(G;Y_1 \times Y_2)\\Z_{13} \in \p_E(G;Y_1 \times Y_3)\\Z_{23} \in \p_E(G;Y_2 \times Y_3)}} \frac{\ab{\Delta(Z_{12} \cup Z_{13} \cup Z_{23})}}{\ab{\Delta(G)}} d(\h\mid Z_{12}\cup Z_{13}\cup Z_{23})^2\\
		&\leq \sum_{\substack{Z_{12} \in \p_E(G;Y_1 \times Y_2)\\Z_{13} \in \p_E(G;Y_1 \times Y_3)\\Z_{23} \in \p_E(G;Y_2 \times Y_3)}} \frac{\ab{\Delta(Z_{12} \cup Z_{13} \cup Z_{23})}}{\ab{\Delta(G)}}\sum_{a,b,c=1}^m \frac{\ab{\Delta(Z_{12}^a\cup Z_{13}^b\cup Z_{23}^c)}}{\ab{\Delta(Z_{12}\cup Z_{13}\cup Z_{23})}} d(\h \mid Z_{12}^a \cup Z_{13}^b \cup Z_{23}^c)^2\\
		&= \sum_{\substack{Z_{12} \in \p_E(G;Y_1 \times Y_2)\\Z_{13} \in \p_E(G;Y_1 \times Y_3)\\Z_{23} \in \p_E(G;Y_2 \times Y_3)}} \sum_{a,b,c=1}^m \frac{\ab{\Delta(Z_{12}^a\cup Z_{13}^b\cup Z_{23}^c)}}{\ab{\Delta(G)}} d(\h \mid Z_{12}^a \cup Z_{13}^b \cup Z_{23}^c)^2\\
		&= \sum_{\substack{Z'_{12} \in \p_E'(G;Y_1 \times Y_2)\\Z'_{13} \in \p_E'(G;Y_1 \times Y_3)\\Z'_{23} \in \p_E'(G;Y_2 \times Y_3)}} \frac{\ab{\Delta(Z'_{12} \cup Z'_{13} \cup Z'_{23})}}{\ab{\Delta(G)}} d(\h\mid Z'_{12}\cup Z'_{13}\cup Z'_{23})^2\\
		&= q(\p_E'(G)),
	\end{align*}
	where in the penultimate equality we simply note that summing over all parts in $\p_E(G)$, and then over all their subparts, is the same as summing over all parts in $\p_E'(G)$. This proves \cref{lemit:G refinement}.

	Note that in the special case that $G$ is complete tripartite, the result just proved shows that $q(\p_E'(K(Y_i,Y_j,Y_k)))\geq q(\p_E(K(Y_i,Y_j,Y_k)))$ whenever $\p_E'(K(Y_i,Y_j,Y_k))$ is an edge partition of a vertex cylinder $K(Y_i,Y_j,Y_k)$ which refines another edge partition $\p_E(K(Y_i,Y_j,Y_k))$.
	
	We now turn to \cref{lemit:P refinement}, so we fix a $t$-partite $3$-graph $\h$ as well as two cylinder chain partitions $\p,\p'$ of it, with $\p'$ refining $\p$.
	Lt $\wh \p$ be the cylinder chain partition whose vertex partition is $\p_V'$, but whose edge partitions are simply obtained by restricting the original edge partition $\p_E$ to the new vertex cylinders. That is, for all $Y' \in \p_V'$, the edge partition $\wh{\p_E}(Y_i' \times Y_j')$ is $\p_E(Y_i \times Y_j)|_{Y_i' \times Y_j'}$, where $Y$ is the unique cylinder in $\p_V$ containing $Y'$. Note that $\p'$ is obtained from $\wh \p$ by only refining the edge partition. 
	As a consequence, we find from the above that
	\begin{align*}
		q(\p') &= \sum_{Y \in \p_V'} \left( \prod_{\ell=1}^t \frac{\ab{Y_\ell}}{\ab{X_\ell}} \right) \sum_{1 \leq i<j <k\leq t} q(\p_E'(K(Y_i,Y_j,Y_k)))\\
		&\geq \sum_{Y \in \p_V'} \left( \prod_{\ell=1}^t \frac{\ab{Y_\ell}}{\ab{X_\ell}} \right) \sum_{1 \leq i<j <k\leq t} q(\wh {\p_E}(K(Y_i,Y_j,Y_k)))\\
		&=q(\wh \p)
	\end{align*}
	Therefore, it suffices to prove that $q(\wh \p) \geq q(\p)$. For this, fix some cylinder $Y \in \p_V$. As $\p_V'$ refines $\p_V$, there are cylinders $Y^1,\dots,Y^m\in \p_V'$ which partition $Y$. 
	For $1 \leq i <j\leq t$ and $Z_{ij} \in \p_E(Y_i \times Y_j)$, let $Z^r_{ij} \in \wh{\p_E}(Y_i^r \times Y_j^r)$ be the restriction of $Z_{ij}$ to the vertex subsets $Y_i^r \cup Y_j^r$. 
	
	For the moment, we focus on $Y_1 \cup Y_2 \cup Y_3$. 
	Note that for any $Z_{12} \in \p_E(Y_1 \times Y_2), Z_{13} \in \p_E(Y_1 \times Y_3), Z_{23} \in \p_E(Y_2 \times Y_3)$, we have
	\[
		\ab{\Delta(Z_{12}\cup Z_{13}\cup Z_{23})} \cdot \prod_{\ell=4}^t \ab{Y_\ell} = \sum_{r=1}^m \ab{\Delta(Z_{12}^r \cup Z_{13}^r \cup Z_{23}^r)} \cdot \prod_{\ell=4}^t \ab{Y_\ell^r} ,
	\]
	as both sides count the number of $t$-tuples in $Y_1 \dtimes Y_t$ in which the first three vertices span a triangle in $Z_{12}\cup Z_{13}\cup Z_{23}$. Similarly, 
	\[
		\ab{E(\h) \cap \Delta(Z_{12}\cup Z_{13}\cup Z_{23})} \cdot \prod_{\ell=4}^t \ab{Y_\ell} = \sum_{r=1}^m \ab{E(\h) \cap \Delta(Z_{12}^r\cup Z_{13}^r\cup Z_{23}^r)} \cdot \prod_{\ell=4}^t \ab{Y_\ell^r} ,
	\]
	as both sides count the number of $t$-tuples in $Y_1 \dtimes Y_t$ in which the first three vertices span a triangle in $Z_{12}\cup Z_{13} \cup Z_{23}$ and form an edge of $\h$.
	Therefore, we have that
	\begin{align*}
		d(\h \mid Z_{12}\cup Z_{13} \cup Z_{23}) &= \frac{\ab{E(\h)\cap \Delta(Z_{12}\cup Z_{13} \cup Z_{23})}}{\ab{\Delta(Z_{12}\cup Z_{13}\cup Z_{23})}}\\
		&= \frac{\ab{E(\h)\cap \Delta(Z_{12}\cup Z_{13} \cup Z_{23})}\cdot \prod_{\ell=4}^t \ab{Y_\ell}}{\ab{\Delta(Z_{12}\cup Z_{13}\cup Z_{23})}\cdot \prod_{\ell=4}^t \ab{Y_\ell}}\\
		&=\sum_{r=1}^m \frac{\ab{E(\h)\cap \Delta(Z_{12}^r\cup Z_{13}^r \cup Z_{23}^r)}\cdot \prod_{\ell=4}^t \ab{Y_\ell^r}}{\ab{\Delta(Z_{12}\cup Z_{13}\cup Z_{23})}\cdot \prod_{\ell=4}^t \ab{Y_\ell}} \\
		&=\sum_{r=1}^m \frac{\ab{\Delta(Z^r_{12}\cup Z^r_{13}\cup Z^r_{23})}\cdot \prod_{\ell=4}^t \ab{Y_\ell^r}}{\ab{\Delta(Z_{12}\cup Z_{13}\cup Z_{23})}\cdot \prod_{\ell=4}^t \ab{Y_\ell}} d(\h\mid Z_{12}^r\cup Z_{13}^r \cup Z_{23}^r).
	\end{align*}
	Again by Cauchy--Schwarz, we conclude that
	\[
		d(\h \mid Z_{12}\cup Z_{13} \cup Z_{23})^2 \leq \sum_{r=1}^m \frac{\ab{\Delta(Z^r_{12}\cup Z^r_{13}\cup Z^r_{23})}\cdot \prod_{\ell=4}^t \ab{Y_\ell^r}}{\ab{\Delta(Z_{12}\cup Z_{13}\cup Z_{23})}\cdot \prod_{\ell=4}^t \ab{Y_\ell}} d(\h\mid Z_{12}^r\cup Z_{13}^r \cup Z_{23}^r)^2.
	\]
	Plugging this in to the definition of $q(\p_E(K(Y_1,Y_2,Y_3)))$, we find that
	\begin{align*}
		q(\p_E(K(Y_1,Y_2,Y_3)))&=\sum_{\substack{Z_{12}\in \p_E(Y_1 \times Y_2)\\Z_{13}\in \p_E(Y_1 \times Y_3)\\Z_{23}\in \p_E(Y_2 \times Y_3)}} \frac{\ab{\Delta(Z_{12}\cup Z_{13}\cup Z_{23})}}{\ab{Y_1}\ab{Y_2}\ab{Y_3}}d(\h \mid Z_{12}\cup Z_{13}\cup Z_{23})^2\\
		&\leq \sum_{\substack{Z_{12}\in \p_E(Y_1 \times Y_2)\\Z_{13}\in \p_E(Y_1 \times Y_3)\\Z_{23}\in \p_E(Y_2 \times Y_3)}} \sum_{r=1}^m \frac{\ab{\Delta(Z^r_{12}\cup Z^r_{13}\cup Z^r_{23})}\cdot \prod_{\ell=4}^t \ab{Y_\ell^r}}{\prod_{\ell=1}^t \ab{Y_\ell}}d(\h\mid Z_{12}^r\cup Z_{13}^r \cup Z_{23}^r)^2\\
		&=\sum_{\substack{Z_{12}\in \p_E(Y_1 \times Y_2)\\Z_{13}\in \p_E(Y_1 \times Y_3)\\Z_{23}\in \p_E(Y_2 \times Y_3)}} \sum_{r=1}^m \left( \prod_{\ell=1}^t \frac{\ab{Y_\ell^r}}{\ab{Y_\ell}} \right)\frac{\ab{\Delta(Z_{12}^r\cup Z_{13}^r\cup Z_{23}^r)}}{\ab{Y_1^r}\ab{Y_2^r}\ab{Y_3^r}}d(\h\mid Z_{12}^r\cup Z_{13}^r \cup Z_{23}^r)^2\\
		&=\sum_{r=1}^m \left( \prod_{\ell=1}^t \frac{\ab{Y_\ell^r}}{\ab{Y_\ell}} \right)\sum_{\substack{\wh{Z_{12}}\in \wh{\p_E}(Y_1^r \times Y_2^r)\\\wh{Z_{13}}\in \wh{\p_E}(Y_1^r \times Y_3^r)\\\wh{Z_{23}}\in \wh{\p_E}(Y_2^r \times Y_3^r)}}\frac{\ab{\Delta(\wh{Z_{12}}\cup \wh{Z_{13}}\cup \wh{Z_{23}})}}{\ab{Y_1^r}\ab{Y_2^r}\ab{Y_3^r}}d(\h\mid \wh{Z_{12}}\cup \wh{Z_{13}} \cup \wh{Z_{23}})^2\\
		&=\sum_{r=1}^m \left( \prod_{\ell=1}^t \frac{\ab{Y_\ell^r}}{\ab{Y_\ell}} \right) q(\wh{\p_E}(K(Y_1^r , Y_2^r , Y_3^r))),
	\end{align*}
	where in the penultimate equality we use the definition of $\wh{\p_E}$, namely that the set of edge parts appearing in it is precisely the set of restrictions of edge parts of $\p_E$.

	By the same argument, the above inequality holds for all indices $1 \leq i<j<k \leq t$. Summing up over all choices for $i,j,k$, we conclude that for any $Y \in \p_V$, we have
	\begin{align*}
		q(\p_E(Y)) &=
		\sum_{1 \leq i<j<k\leq t} q(\p_E(K(Y_i, Y_j, Y_k))) \leq \sum_{1 \leq i<j<k\leq t} \sum_{r=1}^m \left( \prod_{\ell=1}^t \frac{\ab{Y_\ell^r}}{\ab{Y_\ell}} \right) q(\wh{\p_E}(K(Y_i^r , Y_j^r , Y_k^r))).
	\end{align*}
	Therefore, by the definition of $q(\p)$, we have
	\begin{align*}
		q(\p) &= \sum_{Y \in \p_V} \left( \prod_{\ell=1}^t \frac{\ab{Y_\ell}}{\ab{X_\ell}} \right)q(\p_E(Y))\\
		&\leq \sum_{Y \in \p_V}\left( \prod_{\ell=1}^t \frac{\ab{Y_\ell}}{\ab{X_\ell}} \right)
		\sum_{1 \leq i<j<k\leq t} \sum_{r=1}^m \left( \prod_{\ell=1}^t \frac{\ab{Y_\ell^r}}{\ab{Y_\ell}} \right) q(\wh{\p_E}(K(Y_i^r , Y_j^r , Y_k^r)))\\
		&= \sum_{Y \in \p_V} \sum_{r=1}^m \left( \prod_{\ell=1}^t \frac{\ab{Y_\ell^r}}{\ab{X_\ell}} \right) \sum_{1 \leq i<j<k\leq t}
		q(\wh{\p_E}(K(Y_i^r , Y_j^r , Y_k^r)))\\
		&=\sum_{Y' \in \p_V'} \left( \prod_{\ell=1}^t \frac{\ab{Y_\ell'}}{\ab{X_\ell}} \right) \sum_{1 \leq i<j<k\leq t} q(\wh{\p_E}(K(Y_i' , Y_j' , Y_k')))\\
		&=q(\wh{\p}).
	\end{align*}
	As we have already proved that $q(\wh\p) \leq q(\p')$, this concludes the proof of \cref{lemit:P refinement}.
\end{proof}

\subsection{Proof of Lemma \ref{lem:induced counting}}\label{appendix:induced counting}
We begin by recalling the statement of \cite[Corollary 5.3]{MR2373376}, specialized to the case of counting cliques in $3$-uniform hypergraphs.
\begin{theorem}[{Special case of \cite[Corollary 5.3]{MR2373376}}]\label{thm:counting}
	Let $Z$ be a $t$-partite graph on parts $X_1,\dots,X_t$, and let $\h$ be an $t$-partite $3$-graph supported on $\Delta(Z)$.

	Let $\varepsilon, \gamma\in (0, \frac 1{100}]$, and define $\eta=\varepsilon^8\gamma^{10t^3}$ and $\psi(x)=\varepsilon^{10}\gamma^{10t^3}x^{16t^2}$. Suppose that $$d(\h[X;Z;\{i,j,k\}])\geq \gamma$$ and that $\h[X;Z;\{i,j,k\}]$ is $(\eta,\psi)$-quasirandom for all $\{i,j,k\} \in \binom{[t]}3$.

	If $x_1,\dots,x_t \in X_1 \dtimes X_t$ are picked uniformly at random, then the probability that they span a $K_t\up 3$ in $\h$---that is, the probability that $x_ix_jx_k \in E(\h)$ for all $\{i,j,k\} \in \binom{[t]}3$---is at least
	\[
		(1-\varepsilon) \left(\prod_{1 \leq i<j \leq t} d(Z_{ij})\right)  \left(\prod_{1 \leq i<j<k\leq t} d(\h[X;Z;\{i,j,k\}])\right).
	\]
\end{theorem}
\noindent As promised, \cref{lem:induced counting} follows as a straightforward corollary of \cref{thm:counting}.
\begin{proof}[Proof of \cref{lem:induced counting}]
	Let the three vertex parts of $\V$ be $V_1,V_2,V_3$, and let the vertices of $\V$ be $v_1,\dots,v_t$. Let $\iota:\{v_1,\dots,v_t\} \to \{1,2,3\}$ be the ``index'' function that associates to each $v_i$ the unique index $\iota(i)$ such that $v_i \in V_{\iota(i)}$.
	
	We construct a $t$-partite graph $Z$ and a $t$-partite $3$-graph $\h$ as follows. For every $v_i$, we construct a set $X_i$, which is a copy of $U_{\iota(i)}$. Thus, the parts $X_1,\dots,X_t$ consist of $\ab{V_1}$ copies of $U_1$, $\ab{V_2}$ copies of $U_2$, and $\ab{V_3}$ copies of $U_3$.

	Next, we define the $\binom t2$ bipartite graphs making up $Z$ as follows. If $\iota(v_i)\neq \iota(v_j)$, then we define $Z_{ij}$ to be a copy of $G[U_{\iota(i)} \times U_{\iota(j)}]$ between $X_i$ and $X_j$. If $\iota(i)=\iota(j)$, we define $Z_{ij}$ to be the complete bipartite graph between $X_i$ and $X_j$. Note that in this way, we ensure that each $Z_{ij}$ is $\psi(\delta)$-quasirandom, and that each tripartite subgraph satisfies $\delta(Z_{ij} \cup Z_{ik}\cup Z_{jk})\geq \delta$.

	Finally, we define each of the tripartite $3$-graphs making up $\h$. For indices $i,j,k \in \binom{[t]}3$, in case $\ab{\iota(i),\iota(j),\iota(k)}\leq 2$, we include in $\h$ all triangles in the tripartite graph $Z_{ij} \cup Z_{ik} \cup Z_{jk}$. In the remaining case, where $\iota(i),\iota(j),\iota(k)$ are all distinct, we do the following. If $\{v_i,v_j,v_k\} \in E(\V)$, we put a copy of $\h_0$ between $X_i,X_j,X_k$. However, if $\{v_i,v_j,v_k\} \notin E(\V)$, we put between these three parts the \emph{relative complement} of $\h_0$; that is, we include as a hyperedge every triangle in $Z_{ij}\cup Z_{ik} \cup Z_{jk}$ which is not a hyperedge of $\h_0$.
	
	We now claim that $\h[X;Z;\{i,j,k\}]$ is $(\eta,\psi)$-quasirandom with $d(\h[X;Z;\{i,j,k\}])\geq \gamma$ for all $i,j,k$. Indeed, for triples $i,j,k$ with $\ab{\iota(i),\iota(j),\iota(k)}\leq 2$, this is immediate since all triangles in $Z_{ij} \cup Z_{ik} \cup Z_{jk}$ are present. For the remaining triples, this follows from the fact that we assumed $\gamma \leq d(\h_0[U;G])\leq 1-\gamma$, and that a tripartite hypergraph is $\eta$-quasirandom with respect to a tripartite graph if and only if its relative complement is $\eta$-quasirandom with respect to the same tripartite graph.

	Now suppose that we pick $(x_1,\dots,x_t) \in X_1 \dtimes X_t$ uniformly at random. By \cref{thm:counting}, applied with $\varepsilon=\frac{1}{100}$, the probability that they form a $K_t\up 3$ in $\h$ is at least
	\[
		(1-\varepsilon) \left(\prod_{1 \leq i<j \leq t} d(Z_{ij})\right)  \left(\prod_{1 \leq i<j<k\leq t} d(\h[X;Z;\{i,j,k\}])\right) \geq \frac 12 \delta^{\binom t2}\gamma^{\binom t3}.
	\]
	Additionally, for a fixed $i,j$, the probability that $x_i$ and $x_j$ correspond to the same vertex of $V(\h_0)$ is at most $1/{\min \{\ab{U_1},\ab{U_2},\ab{U_3}\}}$, which is at most $\delta^{t^2}\gamma^{t^3}$ by assumption. Hence, the probability that there is any collision among $\{x_1,\dots,x_t\}$ is at most $t^2\delta^{t^2}\gamma^{t^3}$. Therefore, the probability that $x_1,\dots,x_t$ correspond to distinct vertices in $\h_0$ and span a $K_t\up 3$ in $\h$ is at least
	\[
		\frac 12 \delta^{\binom t2}\gamma^{\binom t3}-t^2 \delta^{t^2}\gamma^{t^3}> 0.
	\]
	We conclude that there is a choice of $(x_1,\dots,x_t) \in X_1 \dtimes X_t$ satisfying both of these properties.

	The final observation is that, by our construction of $\h$, these vertices $x_1,\dots,x_t$ correspond to an induced copy of $\V$ in $\h_0$. Indeed, since both $\V$ and $\h_0$ are tripartite, all we need to check is that if $x_i,x_j,x_k$ lie in distinct parts of $\h_0$, then they form a hyperedge of $\h_0$ if and only if $v_i,v_j,v_k$ form a hyperedge of $\V$. But this is immediate from our choice of $\h$, since we included in $\h$ all edges of $\h_0$ when $\{v_i,v_j,v_k\} \in E(\V)$, and only non-edges of $\h_0$ when $\{v_i,v_j,v_k\} \notin E(\V)$. This concludes the proof.
\end{proof}

\end{document}